\documentclass{amsart}
\usepackage[latin1]{inputenc}
\usepackage[T1]{fontenc}
\usepackage{amssymb, amsmath, color, textcomp, url,tikz, mathtools,dsfont}
\usetikzlibrary{matrix,arrows}
\usepackage{courier}
\usetikzlibrary{fit}
\usetikzlibrary{positioning}
\usepackage[labelformat=empty]{caption}

\usepackage[normalem]{ulem}
\usepackage{soul}
\usepackage{mdwlist}

\usepackage{tikz-cd}

\RequirePackage{ifpdf}
\ifpdf
\usepackage[pdftex]{hyperref}
\else
\usepackage[hypertex]{hyperref}
\fi

\usepackage{enumerate,xspace}
\usepackage{chngcntr, csquotes}

\theoremstyle{plain}
\newtheorem{theorem}{Theorem}[section]
\newtheorem{cor}[theorem]{Corollary}
\newtheorem{prop}[theorem]{Proposition}
\newtheorem{lemma}[theorem]{Lemma}

\newcounter{proofcount}
\AtBeginEnvironment{proof}{\stepcounter{proofcount}}% count the proofs

\newtheorem{thm}{Theorem}

\newtheorem*{conj}{Conjecture}

\theoremstyle{remark}
\newtheorem{claim}{Claim}

\newtheorem*{claim*}{Claim}
\makeatletter                  % reset the claim counter each time proofcount is
\@addtoreset{claim}{proofcount}% increased, effectively this restarts the claims
\makeatother                   % at 1 for each proof

\usepackage{pdfpages}

\newenvironment{claimproof}[1][Proof of Claim \theclaim.] 
{%
	\proof[#1]%
	
}
{%
	\endproof%
}

\newenvironment{claimproof*}[1][Proof of Claim.] 
{%
	\proof[#1]%
	
}
{%
	\endproof%
}

\theoremstyle{definition}
\newtheorem{remark}[theorem]{Remark}
\newtheorem{fact}[theorem]{Fact}
\newtheorem{definition}[theorem]{Definition}

\newcounter{substep}
\def\thesubstep{\arabic{substep}}

{\em}

\newcounter{subsubstep}
\def\thesubsubstep{\arabic{subsubstep}}

{\em}

\newcounter{substepC}
\def\thesubstepC{\arabic{substepC}}

{\em}

\newcounter{subsubstepC}[substepC]
\def\thesubsubstepC{\arabic{subsubstepC}}

{\em}

\newcommand{\nc}{\newcommand}

\nc{\Z}{\mathbb{Z}}
\nc{\Q}{\mathbb{Q}}
\nc{\N}{\mathbb{N}}
\nc{\F}{\mathbb{F}}
\nc{\UU}{\mathbb{U}}
\nc{\C}{\mathbb{C}}

\nc{\M}{\mathcal{M}}
\nc{\R}{\mathcal{R}}
\nc{\A}{\mathcal{A}}
\nc{\B}{\mathcal{B}}
\nc\LL{\mathcal L}
\nc\II{\mathcal I}
\nc\E{\mathcal E}

\nc{\stt}{\operatorname{St}}
\nc{\stab}{\operatorname{Stab}}
\nc{\GO}[1]{G_{#1}^{00}}
\nc{\sbgp}[1]{\langle\xspace {#1}\xspace\rangle}
\nc{\Conn}[1]{\langle\xspace {X}\xspace\rangle^{00}_{#1}}
\nc{\band}[1]{\bar d_{\mathcal{#1}}}
\nc\Def{\operatorname{Def}}

\nc{\dcl}{\operatorname{dcl}}

\nc{\acl}{\operatorname{acl}}

\nc\inv{ ^{-1}}
\nc{\im}{\operatorname{im}}
\nc{\tp}{\operatorname{tp}}
\nc\Spec{S^\mathrm{t}}
\nc\HS{S^\mathrm{h}}
\nc\U{\operatorname{U}}
\nc{\cf}{\text{cf.\,}}
\nc{\eg}{\text{e.g. }}

\nc{\Ext}{H^2}

\nc\bM{\overline{M}}
\def\Ind#1#2{#1\setbox0=\hbox{$#1x$}\kern\wd0\hbox to
	0pt{\hss$#1\mid$\hss} \lower.9\ht0\hbox to
	0pt{\hss$#1\smile$\hss}\kern\wd0}
\def\Notind#1#2{#1\setbox0=\hbox{$#1x$}\kern\wd0\hbox to
	0pt{\mathchardef\nn="0236\hss$#1\nn$\kern1.4\wd0\hss}\hbox to
	0pt{\hss$#1\mid$\hss}\lower.9\ht0 \hbox to
	0pt{\hss$#1\smile$\hss}\kern\wd0}

\title[]{Finite central extensions of o-minimal groups}

\address{Departamento de \'Algebra, Geometr\'ia y Topolog\'ia; Facultad de
	Matem\'aticas;
	Universidad Complutense de Madrid; 28040 Madrid, Spain}

\email{ebaro@ucm.es}

\email{dpalacin@ucm.es}

\date{\today}

\author{El\'ias Baro and Daniel Palac\'in}

\thanks{Both authors are supported by the project STRANO PID2021-122752NB-I00 and
	Grupos UCM 910444.}

\begin{document}
	
	\begin{abstract} 
	 We answer in the affirmative a conjecture of Berarducci, Peterzil and Pillay \cite{BPP10} for solvable groups, which is an o-minimal version of a particular case of Milnor's isomorphism conjecture \cite{jM83}. We prove that every abstract finite central extension of a definably connected solvable definable group in an o-minimal structure is equivalent to a definable (hence topological) finite central extension. The proof relies on an o-minimal adaptation of the higher inflation-restriction exact sequence due to Hochschild and Serre. As in \cite{jM83}, we also prove in o-minimal expansions of real closed fields that the conjecture reduces to definably simple groups. 
	\end{abstract}

\maketitle	
	
	\section{Introduction}
A natural question for an infinite group $G$ is whether certain algebraic, topological, or even logical properties are preserved when passing to finite central extensions. For instance, one may ask if every abstract finite central extension of an arbitrary Lie group is a Lie group. This is a particular case of the more general Milnor's isomorphism conjecture \cite{jM83}, which still remains open. Nonetheless, it is known to be true for solvable groups \cite{jM83} and in other cases, see for example \cite{jM83,cS86} as well as \cite{DWS88}. 

In this paper, we aim to study a closely related conjecture concerning definable groups in o-minimal structures (expanding a dense linear order without end-points), due to Berarducci, Peterzil and Pillay \cite{BPP10}. We briefly recall their conjecture.

Let $G$ be an $\emptyset$-definable group in an o-minimal structure $M$ and let $Z$ be a finite abelian group, which we may assume to be $\emptyset$-definable in $M$ after naming its elements if necessary. Following Berarducci, Peterzil and Pillay \cite[Sec.\,1]{BPP10} (see also Section \ref{s:Cocycle} below), we say that an abstract finite central extension 
\[
0 \longrightarrow Z \stackrel{\mu}{\longrightarrow} \widehat G
\stackrel{\pi}{\longrightarrow} G \longrightarrow 1
\]
is {\em naturally interpretable} without parameters if it is equivalent to a central extension
\[
0\longrightarrow Z \stackrel{\mu_0}{\longrightarrow} \widehat G_0
\stackrel{\pi_0}{\longrightarrow} G \longrightarrow 1,
\]
in which the group $\widehat G_0$ and the homomorphisms $\pi_0:\widehat G_0\to G$ are definable in $M$, without additional parameters other than the ones needed to define $G$ and $Z$. Recall that two central extensions $\widehat{G}_1$ and $\widehat{G}_2$ of $G$ by $Z$ are
{\em (abstractly) equivalent} if there is an isomorphism $\gamma:\widehat G_1\to \widehat G_2$ such that the following diagram commutes
\[
\begin{tikzcd}
	0 \arrow{r}  & Z \arrow{r}{\mu_1}\arrow[equal]{d} & \widehat{G}_1      
	\arrow{r}{\pi_1}\arrow{d}{\gamma}        & G \arrow{r}\arrow[equal]{d}  & 1 \\
	0 \arrow{r} & Z \arrow{r}{\mu_2}  & \widehat{G}_2  \arrow{r}{\pi_2}       & G  
	\arrow{r} & 1 
\end{tikzcd}
\]
Notice that a naturally interpretable abstract finite central extension is in fact a topological extension, since every definable group in an o-minimal structure is endowed with Pillay's group topology \cite{aP88}.

Berarducci, Peterzil and Pillay \cite[p.\,474]{BPP10} conjectured the following: every abstract finite central extension of a definably connected definable group $G$  in an o-minimal structure $M$ is naturally interpretable in $M$. They proved it for abelian groups \cite[Thm.\,2.9]{BPP10} and established some equivalent characterisation for definably compact semisimple groups \cite[Thm.\,2.10]{BPP10}. Furthermore, as any group definable in an o-minimal expansion $R$ of the real field $\mathbb R$ is a real Lie group \cite{aP88}, it follows from \cite[Prop.\,2.2]{BPP10} that every abstract finite central extension of a definable  group $G$  in the o-minimal expansion $R$  is naturally interpretable in $R$, whenever $G$ satisfies Milnor's isomorphism conjecture. Consequently, Berarducci, Peterzil and Pillay's conjecture is also satisfied for many groups definable in $R$.

In this paper we solve the conjecture from \cite{BPP10} for solvable groups in the affirmative:

\begin{thm}[Theorem \ref{T:Sol}]\label{T:Sol_Intro}
Let $G$ be a definably connected solvable definable group  in an o-minimal
structure $M$. Then every abstract finite central extension of $G$ is naturally interpretable in $M$. 
\end{thm}

Notice that this is analogous to Milnor's result in \cite{jM83} for solvable Lie groups, and generalises the abelian case from \cite[Thm.\,2.9]{BPP10}. To prove our result it is more convenient to reformulate the o-minimal conjecture in cohomological terms. 

In Section \ref{s:Cocycle} we introduce a definable version $\Ext_d(G,Z)$ of the second \mbox{cohomology} group $\Ext(G,Z)$, which was first considered by Edmundo in \cite{mE03}. We prove that $\Ext_d(G,Z)$ is canonically embeddable into $\Ext(G,Z)$ and it is in one-to-one correspondence with the equivalence classes of naturally interpretable abstract extensions of $G$ by $Z$ (Lemma \ref{L:Natural_Int}). The conjecture from \cite{BPP10} can then be restated as follows: 
\begin{conj}
The natural inclusion $\Ext_d(G,Z)\hookrightarrow \Ext(G,Z)$ is an isomorphism.
\end{conj} 

To approach this isomorphism conjecture we first prove the following result for definable o-minimal groups by considering second cohomology $\Ext_d$ groups. The proof adapts Neeb's proof of \cite[Thm.\,1.5]{kN02}, and the statement corresponds to an o-minimal adaptation of the inflation-restriction exact sequence \cite[Chap.\,III]{HS53} due to Hochschild and Serre.

\begin{thm}[Theorem \ref{T:Ext-Sequence}] Let $A$ be a definably connected group  and let
\[
1\longrightarrow A \stackrel{\iota}{\longrightarrow} B
\stackrel{\beta}{\longrightarrow} C \longrightarrow 1
\]
be an $\emptyset$-definable short exact sequence of groups in an o-minimal structure, where $\iota$ is the inclusion map. Then, for every finite abelian group $Z$ we have
\[
0\longrightarrow \Ext_d(C,Z) \stackrel{\beta_d^*}{\longrightarrow} \Ext_d(B,Z)
\stackrel{\iota_d^*}{\longrightarrow} \Ext_d(A,Z)
\]
is exact, where $\iota_d^*([f]) = [f|_{A\times A}]$ is the restriction map and $\beta_d^*([f]) = [f\circ
(\beta\times \beta)]$ is the inflation.

\end{thm}

This result, together with its abstract counterpart, is then used to prove Theorem \ref{T:Sol}. We also show there that $\Ext(G,Z)\cong Z^{\dim(G/N(G))}$, where $N(G)$ denotes the maximal torsion-free normal subgroup of $G$. 

In the last two sections of the paper we consider the conjecture in o-minimal expansions of real closed fields. In Section \ref{s:Def_Cohom} we study structural properties of $\Ext_d(G,Z)$. Among other results, we show that $\Ext_d(G,Z)$ is finite (Corollary \ref{C:FiniteExt}). Finally, in Section \ref{s:Reduction} we prove the o-minimal analogue of Milnor's result \cite[Lem.\,4]{jM83}. Namely, we show in Corollary \ref{C:Reduction} (cf.\,Theorem \ref{T:Reduction}) that the o-minimal isomorphism conjecture is true for a definably connected definable group $G$ in an o-minimal expansion of a real closed field, if it is true for all definably simple groups.

\section{Generalities on o-minimal finite central extensions}\label{s:Cocycle}

\subsection{Natural interpretability and cocycles}\label{s:NatInt} We briefly recall the concept of naturally interpretable extensions to complement the definition provided in the introduction. Let $G$ be an $\emptyset$-definable group in $M$ and let $Z$ be an abelian finite group which we also assume to be $\emptyset$-definable in $M$ (see Remark \ref{rmk:defZ}). Given an abstract finite central extension 
	\[
0\longrightarrow Z \stackrel{\mu}{\longrightarrow} \widehat G
\stackrel{\pi}{\longrightarrow} G \longrightarrow 1
\]
we can identify it with the three-sorted structure $\mathcal E=(Z,\mu,\widehat G,\pi,M)$ with sorts $Z,\widehat G$ and $M$ and functions for the respective group operations and homomorphisms. 
\begin{definition}\label{def:natinter}
The abstract finite central extension $\mathcal E$ is {\em naturally interpretable} in $M$ without parameters if there are an $\emptyset$-definable group $\widehat{G}_0$ in $M$ and an $\emptyset$-definable homomorphisms $\pi_0:\widehat G_0\to G$ in $M$ such that the structures $\mathcal E$ and $\mathcal E_0=(Z_0,\mu_0,\widehat G_0,\pi_0,M)$ are isomorphic over $M$, {\it i.e.} the isomorphism is the identity on $M$. 
\end{definition}

With the same notation, note that the isomorphism over $M$ between $\mathcal E$ and $\mathcal E_0$  simply means that the following diagram commutes 
\[
\begin{tikzcd}
	0 \arrow{r}  & Z \arrow{r}{\mu}\arrow{d}{\gamma_Z} & \widehat{G}      
	\arrow{r}{\pi_1}\arrow{d}{\gamma}        & G \arrow{r}\arrow[equal]{d}  & 1 \\
	0 \arrow{r} & Z_0 \arrow{r}{\mu_0}  & \widehat{G}_0  \arrow{r}{\pi_0}       & G  
	\arrow{r} & 1 
\end{tikzcd}
\]
Observe that, without loss of generality, we may assume that $Z_0=Z$ by replacing the finite central extension $\mathcal E_0$ by $(Z,\mu_0\circ\gamma_Z,\widehat G_0,\pi_0,M)$. 

\begin{remark}\label{rmk:defZ}
The diagram above implies that in $M$ one can define the group $\ker(\pi_0$), which is isomorphic to the finite group $Z$. This is the reason we impose the definability of $Z$ from the beginning. Otherwise, if $Z$ is not isomorphic to a definable group in $M$, then the abstract finite central extension $\mathcal E$ cannot be naturally interpretable in $M$. 

Nevertheless, note that there is no harm in assuming that $Z$ is $\emptyset$-definable in $M$, as we can always name its elements with constants. Moreover, if $M$ is an expansion of a real closed field, then every element of an abstract finite group can be identified with $\emptyset$-definable elements of $M$, so the finite group can be taken as $\emptyset$-definable. Thus, no parameters (or new constants) are needed whenever $G$ is $\emptyset$-definable in $M$.
\end{remark}

\begin{remark}\label{R:Parameters}
In the definition of naturally interpretable it is also reasonable to allow the group $\widehat G_0$ and the homomorphism $\pi_0$ to be definable with parameters, say over a set $B$. In this case, we say that $\mathcal E$ is naturally interpretable with parameters (or over $B$). Nonetheless, in this paper we only consider natural interpretability without parameters. In fact, when $M$ is an o-minimal expansion of a real closed field and $G$ and $Z$ are definable without parameters, by \cite[Cor.\,1.2]{EJP11}, the notions of naturally interpretable with and without parameters agree.
\end{remark}

As exposed in the introduction, it will be convenient to consider second homology groups and treat central group extensions via $2$-cocycles. 

Recall that a central group extension of $G$ by an abelian group $Z$ is given by a {\em $2$-cocycle}, which is a map  $f:G\times G\to Z$ that satisfies $f(x,1) = 0
	= f(1,x)$ and
	\[
	f(x,y) + f(xy,z) = f(x,yz) + f(y,z).
	\]
	For a $Z$-valued $2$-cocycle $f$ of $G$, we can define the group
	$G\times_f Z$, whose domain is $G\times Z$ and its operation is defined as
	follows:
	\[
	(x_1,z_1) (x_2,z_2) = ( x_1x_2, z_1 +z_2  +f(x_1,x_2)).
	\]
	This yields a central extension $0 \to Z \stackrel{\mu}{\to} G\times_f Z
	\stackrel{\pi}{\to} G \to 1$ by setting $\pi:G\times_f Z\to G$ to be the
	projection onto the first coordinate and $\mu:Z\to G\times_f Z$ the natural
	inclusion map. In fact, every central group extension 
	\[
	0\longrightarrow Z \stackrel{\mu}{\longrightarrow} \widehat G
	\stackrel{\pi}{\longrightarrow} G \longrightarrow 1
	\]
	is  equivalent to the group extension given by $G\times_{f_s} Z$, where $f_s:G\times
	G\to Z$ is a $2$-cocycle satisfying $\mu(f_s(x,y)) = s(xy)\inv s(x)s(y)$ where
	$s:G\to \widehat{G}$ is a section with $s(1) =1$. Note that the equivalence is
	given by the isomorphism $\gamma: G\times_{f_s} Z\to \widehat{G}$ defined by $(x,z) \mapsto 	s(x)\cdot \mu(z)$. Note that if the section $s:G\to \widehat{G}$ is also a homomorphism, then $f_s$ is the trivial map and moreover $s(G)$ is a normal subgroup $\widehat G$ as $Z$ is central. Hence, the groups $\widehat G$ and $G\times Z$ are isomorphic, in which case $\widehat G$ splits trivially.
	
	When the group $G$ is definable in a structure $M$, we have the following relation between naturally interpretable extensions and 	definable $2$-cocycles (cf.\,\cite[Sec.\,3.1]{BPP10}). %Notice that we always may assume that $Z$ is definable in $M$, since it is finite. Furthermore, after adding constants to the language if necessary, we may assume that $Z$ is $\emptyset$-definable.
	
	\begin{lemma}\label{L:Natural_Int}
		Let $G$ be a definably connected $\emptyset$-definable group in an o-minimal structure $M$ and let
		\[
		0\longrightarrow Z \stackrel{\mu}{\longrightarrow} \widehat G
		\stackrel{\pi}{\longrightarrow} G \longrightarrow 1
		\]
		be an abstract finite central extension, with $Z$ definable in $M$ without parameters. The finite central extension is naturally
		interpretable in $M$ without parameters if and only if there is an $\emptyset$-definable $2$-cocycle $f:G\times
		G\to Z$ such that $\widehat{G}$ and $G\times_f Z$ determine  equivalent extensions. 	
	Moreover, this equivalence is given by a definable isomorphism whenever the group extension given by $\widehat{G}$ is $\emptyset$-definable.
	\end{lemma}
	\begin{proof} 
		Suppose that the group extension  given by $\widehat G$ is naturally
		interpretable. Without loss of generality, we clearly may assume that the extension itself is
		definable in $M$ without parameters. So, by \cite[Thm.\,3.10]{mE03} we obtain an
		$\emptyset$-definable section $s:G\to \widehat{G}$ with $s(1)=1$, and let
		$f_s:G\times G\to Z$ be the associated cocyle. Note that the group  $G\times_{f_s} Z$ is
		definable without parameters as well. Hence, since as explained above the
		extensions $\widehat{G}$ and $G\times_{f_s} Z$ are equivalent, we obtain the result
		as the maps associated to the extension $G\times_{f_s} Z$ are clearly definable.

		For the other direction, notice that any extension equivalent to $G\times_f Z$
		for some $\emptyset$-definable $2$-cocycle $f$ is clearly naturally interpretable  by definition. So, there is some function $\gamma:G\times_f Z\to \widehat G$ witnessing the desired isomorphism over $M$.
	\end{proof}

\subsection{Definable second cohomology group}	
	
	Let $G$ be a group and let $Z$ be a finite abelian group. As usual, we denote by $Z^2(G,Z)$ the group of $Z$-valued $2$-cocycles of $G$, that is,
	\[
	Z^2(G,Z) = \left\{ f:G\times G\to Z  \ | \ f \text{ is a $2$-cocycle on
		$G$}\right\},
	\]
	whose operation is given by pointwise addition. So, it is clearly an 	abelian group, as so is $Z$.	
	For a map $h:G\to Z$ with $h(1)=0$ we always obtain, using the terminology from \cite[Chap.\,16]{tT14}, 
	the $2$-cocycle 
	\[
	dh(x,y)=h(xy)-h(x)-h(y).
	\]
	Cocycles of this form are called {\em $2$-coboundaries}, {\it i.e.} an
	element $f\in Z^2(G,Z)$ is a {\em $2$-coboundary} if there is some $h:G\to Z$
	such that $h(1)=0$ and $f=dh$.  
	The set of $2$-coboundaries $B^2(G,Z)$ is a subgroup of $Z^2(G,Z)$. It is
	routine to verify that two extensions $G\times_{f_1} Z$ and $G\times_{f_2} Z$
	are equivalent if and only if $f_1-f_2\in B^2(G,Z)$, see for example
	\cite[Thm.\,7.32]{jR95}. In particular, an extension $G\times_f Z$ splits trivially ({\it i.e.} is equivalent to $G\times Z$) if and only if $f$ is a $2$-coboundary. Therefore, we will be interested in the second cohomology group
	\[
	\Ext(G,Z) = Z^2(G,Z) / B^2(G,Z). 
	\]
	In particular, the class of a $2$-cocycle determines a central extension up to equivalence, and viceversa.
	
	In the light of the Lemma \ref{L:Natural_Int}, we introduce the following definable version of $\Ext$ to 
	study naturally interpretable extensions. 
	\begin{definition}
		Let $G$ be an $\emptyset$-definable group in an structure $M$ and let $Z$ be finite group, which is also $\emptyset$-definable in $M$.  We write
		$Z_d^{2}(G,Z)$ to denote the subgroup of {\em $\emptyset$-definable $2$-cocycles}, and
		$B_d^2(G,Z)$ for the subgroup of $2$-coboundaries $f\in B^2(G,Z)$ such that $
		f=dh$ for some $\emptyset$-definable function $h:G\to Z$ with $h(1)=0$. Define the {\em definable second cohomology group} as \[
		\Ext_d(G,Z) = Z_d^{2}(G,Z) / B_d^{2}(G,Z).
		\]
	\end{definition}
	We will see in Corollary \ref{C:Ext-Inclusion} that $\Ext_d(G,Z)$ is the group of central extensions of $G$ by $Z$ which are naturally interpretable, by Lemma \ref{L:Natural_Int}.	
	
	We need the following result of Berarducci, Peterzil and Pillay \cite{BPP10} which is hidden in the proof of the claim of Theorem 2.10,
	see also \cite[Prop.\,2.11]{oF17}. We will apply it without explicit mention in the subsequent results. Since it will be extremely useful
	in the sequel, we record a proof for the sake of completeness.
	
	\begin{fact}\label{F:Connected}
		Let $G$ be a definably connected group in an o-minimal structure. Then $G$ has
		no finite-index abstract proper subgroup.
	\end{fact}
	\begin{proof}
		Let $H$ be a finite-index proper subgroup of $G=G^\circ$, which we clearly may
		assume that it is normal. Set $n=|G/H|$ and consider the map $\sigma_n:G\to G$
		given by $\sigma_n(x)=x^n$. By \cite[Lem.\,8.1]{HPP11} we have that $G=\langle
		\im(\sigma_n)\rangle$. On the other hand, the image $\im(\sigma_n)\subset H$, so
		$G=\langle \im( \sigma_n)\rangle \le H$, which yields the statement. 
	\end{proof}
	
	\begin{lemma}\label{L:h_def}
		Let $G$ be a definably connected group in an o-minimal structure and let $Z$ be
		an abelian group of order $n$. If $f\in B^{2}(G,Z)$ is a $2$-coboundary, then
		there is a unique map  $h:G\to Z$ with $h(1)=0$ and $f=dh$. Moreover, for every $x\in G$
		there are some $y_1,\ldots,y_r$ satisfying  $x=y_1^n\cdots y_r^n$ such that
		\[
		h(x)  = \sum_{j=1}^{r} \left( f(y_1^n\cdots y_{j-1}^n,y_j^n) + \sum_{i=1}^{n-1}
		f(y_j,y_j^i) \right).
		\]
		In particular, if $f$ is $\emptyset$-definable, then so is $h$ and hence $f\in
		B_d^2(G,Z)$. 
		
	\end{lemma}
	\begin{proof}
		As $f$ is a $2$-coboundary, by definition there is some $h:G\to Z$ with $h(1)=0$
		such that $dh=f$. Observe that $h$ must be unique. Indeed, suppose that
		$h_0:G\to Z$ is such that $dh_0=f$ and set $\tilde h=h_0-h$. Since $dh_0=f=dh$,
		we have for every $x,y\in G$ that
		\[
		0=dh_0(x,y)-dh(x,y) = d\tilde h(x,y)=\tilde h(xy) -\tilde h(x) - \tilde h(y).
		\]
		This yields that $\tilde h:G\to Z$ is an homomorphism and by the connectedness
		of $G$ we deduce, using Fact \ref{F:Connected}, that $\tilde h=0$, so $h=h_0$.
		
		Now, let $n$ be the order of $Z$. Using the fact that $f(x,y)=h(xy)-h(x)-h(y)$
		for every $x,y\in G$, we easily see that 
		\[
		\sum_{i=1}^{n-1} f(y,y^i) = \sum_{i=1}^{n-1} \left( h(y^{i+1}) - h(y) -
		h(y^i)\right)  = h(y^n),
		\]
		since $nh(y)=0$. As $G=G^\circ$, by \cite[Lem.\,8.1]{HPP11} there is some
		natural number $r\ge 1$ such that for every element $x\in G$ there are
		$y_1,\ldots,y_r\in G$ satisfying $x=y_1^n\cdots y_r^n$. A straightforward recursive argument on $r$ shows that
		\begin{align*}
			h(x)  = h(y_1^n\cdots y_r^n ) & = \sum_{j=1}^{r} \left(f(y_1^n\cdots
			y_{j-1}^n,y_j^n) +  h(y_j^n) \right) \\ &
			= \sum_{j=1}^{r} \left( f(y_1^n\cdots y_{j-1}^n,y_j^n) + \sum_{i=1}^{n-1}
			f(y_j,y_j^i) \right)  ,
		\end{align*}
		which yields that the graph of $h$ is definable.
	\end{proof}

	\begin{cor}\label{C:Ext-Inclusion} Let $G$ be a definably connected $\emptyset$-definable group in an
		o-minimal structure and let $Z$ be a finite abelian group, which we assume to be $\emptyset$-definable. We have that
		$B_d^2(G,Z) = Z_d^2(G,Z) \cap B^2(G,Z)$ and the map $\Ext_d(G,Z)\hookrightarrow
		\Ext(G,Z)$ given by $[f] \to [f]$ is an injective homomorphism. \qed
	\end{cor}

An immediate translation of the corollary above extends \cite[Fact\,2.7]{HPP11} to arbitrary o-minimal groups, nonetheless it is not used in the sequel. 
	
	\begin{cor}
		Let $G$ be a definably connected definable group in an o-minimal structure and let $Z$ be
		a definable finite abelian group. If a definable group extension 
		of $Z$ by $G$ splits trivially as an abstract extension, then it is 		splits trivially and definably.
	\end{cor}

	\begin{proof}
	We may assume that both $G$ and $Z$ are definable without parameters and also that the extension is given by a $2$-cocycle $f\in Z_d^2(G,Z)$, by Lemma \ref{L:Natural_Int}. Note that $f\in B^2(G,Z)$ since the extension splits trivially, as an abstract extension. So, by Lemma \ref{L:h_def}, we have that $f=dh$ for a definable map $h:G\to Z$ with
		$h(1)=0$. Hence, the homomorphism $\alpha:G\to G\times_f
		Z$ defined by $\alpha(g) = (g,h(g))$ is definable. It then follows that
		$G\times_f Z$ can be written as a direct product of $\im(\alpha)$ and the subgroup
		$\{1\}\times Z$, which are both definable. 
	\end{proof}

	From now on we write $\Ext_\diamond$ to refer
	indistinctly to both groups  $\Ext$ and $\Ext_d$.
	Likewise, we use the terminology $Z_\diamond^2$ and  $B_\diamond^2$.

\subsection{A definable inflation-retraction}
We now prove the following result on exact sequences for o-minimal groups. The proof is an adaptation of \cite[Thm.\,1.5]{kN02} (cf.\,\cite[p.\,354]{sL63} and \cite[Thm.\,51.3]{lF70}), where the statement holds for a suitable subgroup of $\Ext(B,Z)$ under the assumption that $A$ is a central subgroup of $B$. It is worth noticing that in the statement below we need only impose that $A$ has no finite-index proper subgroup, which permits to adapt Neeb's proof.

\begin{theorem}\label{T:Ext-Sequence}
Let $A$ be a definably connected group  and let
		\[
		1\longrightarrow A \stackrel{\iota}{\longrightarrow} B
		\stackrel{\beta}{\longrightarrow} C \longrightarrow 1
		\]
		be an $\emptyset$-definable  short exact sequence in an o-minimal structure $M$, where $\iota$ is the inclusion map. Let $Z$ be a finite abelian group and suppose it is $\emptyset$-definable (after naming its elements). Then
		\[
		0\longrightarrow \Ext_\diamond(C,Z) \stackrel{\beta_\diamond^*}{\longrightarrow} \Ext_\diamond(B,Z)
		\stackrel{\iota_\diamond^*}{\longrightarrow} \Ext_\diamond(A,Z)
		\]
		is exact, where $\iota_\diamond^*([f]) = [f|_{A\times A}]$  and $\beta_\diamond^*([g]) = [g\circ
		(\beta\times \beta)]$ are the restriction and  the inflation maps, respectively.
	\end{theorem}
The abstract version of the statement for arbitrary groups is called the higher inflation-restriction exact sequence \cite[p.\,129]{HS53}. Thus, it only remains to prove Theorem \ref{T:Ext-Sequence} for $H_d^2$. Nonetheless, our proof is uniform for both the abstract and the definable case. 
	\begin{proof}
		%We follow the lines of the proof of \cite[Theorem 1.5]{kN02}, emphasizing on
		%definable aspects. 
		It is routine to verify that the functions $\iota_\diamond^*$ and
		$\beta_\diamond^*$ are  well-defined. We prove that the sequence is exact. 
		
		We prove that $\beta_\diamond^*$ is injective. Let $g\in Z^2_\diamond(C,Z)$ be a
		$2$-cocycle for which there is some (definable) map $h:B\to Z$ such that
		$g(\beta(x),\beta(y))=h(xy)-h(x)-h(y)$ for every $x,y\in B$
		and with $h(1)=0$. In particular, for $x,y\in A$ it holds that
		\[
		0=g(1,1)=g(\beta(x),\beta(y)) = h(xy) - h(x)-h(y),
		\]
		so $h|_A:A\to Z$ is a group homomorphism. Moreover, as $Z$ is finite, so is the
		index of $\ker(h|_A)$ in $A$ and hence $h|_A$ is the trivial map, since $A$ has
		no finite-index proper subgroup. As a consequence, for every $x\in A$ and $y\in
		B$ arbitrary we have
		\[
		0=g(1,\beta(y))=g(\beta(x),\beta(y)) = h(xy) -h(y).
		\]
	In particular, the function $\tilde{h}:C\rightarrow Z$ defined by $\tilde{h}(c)=h(x)$ for some $x\in \beta\inv(c)$ is well-defined. We then clearly have that $g=d\tilde{h}$, so that $[g]=[0]$, as desired.

		Now, we prove that $\ker(\iota_\diamond^*)= \im(\beta_\diamond^*)$. Observe that
		$(\iota_\diamond^*\circ \beta_\diamond^*)([f])=[0]$ for every $g\in
		Z_\diamond^2(C,Z)$, so $\im(\beta_\diamond^*) \subset \ker(\iota_\diamond^*)$.
		Hence, it remains to prove $\ker(\iota_\diamond^*)\subset
		\im(\beta_\diamond^*)$.
		
		Let $f\in Z_\diamond^2(B,Z)$ and suppose that $\iota_\diamond^*([f])=[0]$, {\it
			i.e.}, there is some function $h_0:A\to Z$ such that $h_0(1)=0$ and
		$f(x,y)=dh_0(x,y)=h_0(xy)-h_0(x)-h_0(y)$ for every $x,y\in A$. Set
		$\widehat{B}=B\times_f Z$ and let 
		\[
		0\longrightarrow Z \stackrel{\mu}{\longrightarrow} \widehat{B}
		\stackrel{\pi}{\longrightarrow} B\longrightarrow 1
		\]
		be the group extension where $\mu$ and $\pi$ are the natural maps.
		
		We first construct a group extension of $C$ by $Z$. Consider the normal subgroup
		$\widehat A=A\times_f Z$ of $\widehat B$. Note that there is a group
		monomorphism $\alpha:A\hookrightarrow \widehat A$ via the map $a\mapsto
		(a,h_0(a))$, since $f(x,y)=h_0(xy)-h_0(x)- h_0(y)$ for every $x,y\in A$, and moreover
		the image subgroup $\alpha(A)$ has finite index in $\widehat{A}$. Now, since $A$
		has no finite-index proper subgroup, neither does $\alpha(A)$. Also, it is a
		finite-index subgroup of $\widehat A$, so we obtain that $\alpha(A)$ is a
		characteristic subgroup of $\widehat A$. Therefore, as $\widehat A$ is normal in
		$\widehat{B}$, we deduce that $\alpha(A)$ is a normal subgroup of $\widehat{B}$.
		Thus, we can consider the quotient group $\widehat C=\widehat{B}/\alpha(A)$. 
		
		We now define the group homomorphism $\pi_0: \widehat{C} \to C$ by
		$[(x,z)] \mapsto \beta(x)$ and note that 
		\[
		0\longrightarrow Z \stackrel{\mu_0}{\longrightarrow} \widehat{C}
		\stackrel{\pi_0}{\longrightarrow} C\longrightarrow 1
		\]
		is a finite central extension, where $\mu_0(z) = [(1,z)]$. Thus, there exists
		some $2$-cocycle $g:C\times C\to Z$ and some isomorphism $\gamma:\widehat C\to C\times_g Z$ such that the following diagram commutes
		\[
		\begin{tikzcd}
			0 \arrow{r}  & Z \arrow{r}{\mu_0}\arrow[equal]{d} & \widehat{C}      
			\arrow{r}{\pi_0}\arrow{d}{\gamma}        & C \arrow{r}\arrow[equal]{d}  & 1 \\
			0 \arrow{r} & Z \arrow{r}{\mu_1}  & C\times_{g} Z  \arrow{r}{\pi_1}       & C  
			\arrow{r} & 1 
		\end{tikzcd},
		\]
		where $\mu_1(z)=(1,z)$ and $\pi_1(x,z)=x$. 	
		Since $\pi_1\circ \gamma =\pi_0$ and $\gamma\circ \mu_0=\mu_1$, 
		%we have that $\gamma([x,0]) = (\beta(x),z_x)$ for some $z_x\in Z$, and
		%$\gamma([(x,z)]) = (1,z)$. So, 
		we obtain $\gamma([(x,z)]) = (\beta(x),z+z_x)$ for some $z_x\in Z$ which depends on
		$x$, with $z_1=0$. So, we can naturally define a map $h:B\to Z$ as $h(x)=z_x$. Now, an easy
		computation yields, on one hand that
		\begin{align*}
		\gamma\left( [(x_1,z_1)][(x_2,z_2)] \right) 
		&= \gamma\left( [(x_1x_2,z_1 + z_2 + f(x_1,x_2))] \right)  \\
	&	= \left( \beta(x_1x_2) , z_1 + z_2 + f(x_1,x_2) + h(x_1x_2) \right),
		\end{align*}
		and on the other
		\begin{align*}
			\gamma\left( [(x_1,z_1)]\right)\gamma\left([(x_2,z_2)] \right) & =  \left( \beta(x_1) , z_1 + h(x_1) \right) 
			\left( \beta(x_2) , z_2 + h(x_2) \right) \\ 
			&  = \left( \beta(x_1x_2) , z_1 + z_2 + g(\beta(x_1),\beta(x_2)) + h(x_1) +
			h(x_2) \right).
		\end{align*}
		This yields that $f(x_1,x_2) = g(\beta(x_1),\beta(x_2)) - dh(x_1,x_2)$.
		Therefore, we have proven that  $\beta^*([g]) = [f]$, as desired. 
		
		Finally, observe that when $f$ is $\emptyset$-definable, then so are $\widehat{B}$ and the
		map $h_0:A\to Z$ by Lemma \ref{L:h_def}. Hence, the group $\alpha(A)$ and so
		$\widehat{C}$, as well as the above group extension
		\[
		0\longrightarrow Z \stackrel{\mu_0}{\longrightarrow} \widehat{C}
		\stackrel{\pi_0}{\longrightarrow} C\longrightarrow 1
		\]
		are all $\emptyset$-definable in the o-minimal structure $M$. In particular, by Lemma
		\ref{L:Natural_Int} we can take the $2$-cocycle $g\in Z^2(C,Z)$  to be 
		$\emptyset$-definable, so $\beta_d^*([g]) = [f]$. This finishes the proof.
	\end{proof}

	Altogether, we obtain the following corollary for definable group extensions in
	o-minimal structures, by Corollary \ref{C:Ext-Inclusion}
	and Theorem \ref{T:Ext-Sequence}.
	
	\begin{cor}\label{C:Key_Diagram}
		Let $1\to A \stackrel{\iota}{\to} B \stackrel{\beta}{\to} C \to 1$
		be a short exact sequence of groups which is $\emptyset$-definable in an o-minimal
		structure, where $\iota$ is the inclusion map. Assume further that $A$ is
		definably connected and let $Z$ be a finite abelian group which is $\emptyset$-definable. Then the following diagram commutes
		\[
		\begin{tikzcd}
			0 \arrow{r}  & \Ext(C,Z) \arrow{r}{\beta^*} & \Ext(B,Z)      \arrow{r}{\iota^*}
			& \Ext(A,Z) \\
			0 \arrow{r} & \Ext_d(C,Z) \arrow{r}{\beta_d^*}\arrow[hookrightarrow]{u} &
			\Ext_d(B,Z)      \arrow{r}{\iota_d^*}\arrow[hookrightarrow]{u}  & \Ext_d(A,Z)   
			\arrow[hookrightarrow]{u} 
		\end{tikzcd}
		\]
		where $\iota_\diamond^*([f]) = [f|_{A\times A}]$ and $\beta_\diamond^*([g]) =
		[g\circ (\beta\times \beta)]$, and both rows are exact. \qed
	\end{cor}

We finish this section with another adaptation of the inflation-retraction exact sequence to the o-minimal setting, see \cite[6.8.3]{cW94} or \cite[p.\,354]{sL63}.

\begin{lemma}\label{L:Inf-Ret} Let $1\to A \stackrel{\iota}{\to} B \stackrel{\beta}{\to} C \to 1$
	be a short exact sequence of groups which is $\emptyset$-definable in an o-minimal
	structure, where $\iota$ is the inclusion map and $A$ is finite. Assume that $B$ and $C$ are definably connected. Let $\omega\in Z_d^2(C,A)$ be a $2$-cocycle corresponding to the given short sequence.	Then for each $\emptyset$-definable finite abelian group $Z$ we have an exact sequence
	\[
	0\longrightarrow \mathrm{Hom}(A,Z)\stackrel{\omega_*}{\longrightarrow} \Ext_{\diamond}(C,Z)\stackrel{\beta_\diamond^*}{\longrightarrow} \Ext_\diamond(B,Z)
	\]	
	where $\omega_*$ denotes the connecting homomorphism $\omega_*(\alpha) = [\alpha\circ \omega]$. 
\end{lemma}	
\begin{proof}
Assume, as we may, that $B=C\times_\omega A$ and $\beta:B\to C$ is the projection on the first coordinate.	Observe first that $\alpha\circ \omega$ is a definable $2$-cocycle in $Z^2_d(C,Z)$, since $A$ and $Z$ are finite and so $\alpha:A\to Z$ is $\emptyset$-definable. Thus, the map $\omega_*$ is well-defined. Moreover, $\omega_*$ is injective. For, suppose that $\omega_*(\alpha)=[0]$, so there is some definable $h':C\to Z$ with $h'(1)=0$ such that $\alpha\circ \omega=dh'$. Define the map $B\to Z$ by $(c,a)\mapsto \alpha(a) - h'(c)$ and note that it is an homomorphism. So, it must be trivial and hence $\alpha(a)=0$ for every $a\in A$, as desired.	

%Also, it is clear that $[\alpha\circ f]=[0]$ if and only if $\alpha$ is the map $a\mapsto 0$.

We prove the exactness of the sequence at $\Ext_\diamond(C,Z)$. To see $\im(\omega_*)\subset \ker(\beta_\diamond^*)$, note that $\alpha\circ\omega \circ (\beta \times \beta)=dh_0$ for a definable $h_0:B\to Z$ given by $h_0(c,a)=\alpha(a)$, so $\omega_*(\alpha)\in\ker(\beta_\diamond^*)$. 

On the other hand,  if $g\circ(\beta\times\beta)=dh$ for $g\in Z_\diamond^2(C,Z)$ and some (definable) $h:B\to Z$ with $h(1,0)=0$, then set $\alpha(a)=h(1,a)$ and note that
\[
0=g(1,1)=dh((1,a_1),(1,a_2)) = \alpha(a_1+a_2)-\alpha(a_1)-\alpha(a_2).
\]
So, we have $\alpha\in\mathrm{Hom}(A,Z)$. It then follows for $c\in C$ and $a\in A$ that
\[
0=g(c,1)=dh((c,0),(1,a))=h(c,a)-h(c,0)-h(1,a),
\]
which implies that $h(c,a)=h(1,a)+h(c,0)=\alpha(a)+h(c,0)$. Thus, if we define $h_1:C\rightarrow Z$ by $c \mapsto h(c,0)$, then for $c_1,c_2\in C$ we obtain 
$$
\begin{array}{rcl}
	
g(c_1,c_2) &= & dh((c_1,0),(c_2,0)) = h((c_1c_2,\omega(c_1,c_2))-h(c_1,0)-h(c_2,0) \\
&= & \alpha(\omega(c_1,c_2))+h(c_1c_2,0)-h(c_1,0)-h(c_2,0) \\
& = & \alpha(\omega(c_1,c_2))+dh_1(c_1,c_2),
\end{array}$$
yielding that $\omega_*(\alpha)=[g]$ and hence $ \ker(\beta_\diamond^*)\subset \im(\omega_*)$.\end{proof}
	
\section{Solvable o-minimal groups}
	
	In this section we prove our main theorem, which is deduced from the
	general structure of solvable o-minimal groups via Corollary
	\ref{C:Key_Diagram}. As a first step towards the proof, we need  the following
	result on the second cohomology group of torsion-free groups. Recall that
	every torsion-free group defined in an o-minimal structure is solvable, by
	\cite[Clm.\,2.11]{PS05}. Furthermore, Strzebonski proved in
	\cite[Lem.\,2.5]{aS94} that these groups are precisely those definable groups
	with Euler characteristic $\pm 1$ and hence they are uniquely divisible, by
	\cite[Prop.\,4.1]{aS94}.
	
	\begin{prop}\label{P:Torsion-free}
		Let $G$ be a torsion-free $\emptyset$-definable group in an o-minimal structure and let $Z$
		be an $\emptyset$-definable finite abelian group. Then $\Ext(G,Z)=0$.
	\end{prop}
	\begin{proof}
		As we have just remarked, the group $G$ is solvable and uniquely divisible. In
		particular, it is definably connected. To see that $\Ext(G,Z)=0$ we first prove
		it for abelian groups.
		
		\begin{claim*}
			If $A$ is a torsion-free abelian group definable in an o-minimal structure and
			$Z$ is a finite abelian group, then $\Ext(A,Z)=0$. 
		\end{claim*}
		\begin{claimproof*}
			In this case, since $A$ is definably connected, every abstract finite central extension
			of $A$ is abelian by \cite[Lem.\,2.8]{BPP10}. %Indeed, since $\widehat{A}/Z\cong A$ is abelian, we have that $\widehat{A}$ is nilpotent of class $2$. Now, for any element $a\in\widehat{A}$, consider the group homomorphism $\alpha_a:A\to Z$ given by $x\mapsto [y,a]$, where $y\in \widehat{A}$ satisfies $x=\pi(y)$. As $Z$ is finite and $A$ divisible, we have that $A=\ker(\alpha_a)$ and so $\widehat{A}= C_{\widehat{A}}(a)$. Therefore, the group $\widehat{A}$ is abelian as $a\in \widehat{A}$ was arbitrary.
			So, the result follows from 
			\cite[pp.\,222-223]{lF70}. Nonetheless, we give a shorter proof for the sake of
			completeness.
			
			Write $A$ additively, set $n=|Z|$ and let $0\longrightarrow Z
			\stackrel{\mu}{\longrightarrow} \widehat A \stackrel{\pi}{\longrightarrow} A
			\longrightarrow 0$ be an abstract finite central group extension. Define $\psi:A\to
			\widehat A$ as $\psi(a)=nx$ where $x\in \widehat{A}$ satisfies that $\pi(x)$ is
			the unique element of $A$ such that $n\pi(x)=a$. Note that this is a
			well-defined homomorphism. Moreover, by noticing that $\pi\circ
			\psi=\mathrm{id}_A$, we get that $\ker(\psi)=0$. Once $\psi$ is a monomorphism and a section, we are done. To see this, define the homomorphism
			\[
			\gamma: A \times Z \to \widehat{A}, \ (a,z) \mapsto \psi(a)+\mu(z).
			\]
			Note that $\widehat A=\im(\psi) +
			\mu(Z)$, so the above homomorphism is a surjection. Furthermore, its kernel is
			trivial. Indeed, if $(a,z)\in \ker(\gamma)$, then $\pi(\psi(a))=\pi(-\mu(z))=0$,
			which implies that $a=0$ and hence $z=0$. So $A\times Z$ and
			$\widehat{A}$ are isomorphic, and it is routine to verify that both extensions
			are equivalent, by considering the natural maps. Therefore, we have seen that
			every finite central extension of $A$ splits trivially, so $\Ext(A,Z)=0$.
		\end{claimproof*}
Once we know the result for torsion-free abelian groups, we proceed by induction on the dimension of $G$, being the case $\dim(G)=0$ trivial.
		
Suppose that $G$ is infinite. As it is solvable, its commutator subgroup $G'$ is proper, definable and definably connected by \cite[Cor.\,6.8]{BJO12}. Thus, it has no 		finite-index proper subgroup, by Fact \ref{F:Connected}. Furthermore, we can consider the exact sequence 
\[		
1\longrightarrow G' \stackrel{\iota}{\longrightarrow} G		\stackrel{\beta}{\longrightarrow} G/G' \longrightarrow 1	
\]		
and so by Theorem \ref{T:Ext-Sequence} we obtain the exact sequence 
\[		
0\longrightarrow \Ext(G/G',Z) \stackrel{\beta^*}{\longrightarrow} \Ext(G,Z) \stackrel{\iota^*}{\longrightarrow} \Ext(G',Z).		
\]
The inductive hypothesis yields that $\Ext(G',Z)=0$ and we readily get that the map $\beta^*: \Ext(G/G',Z) \to \Ext(G,Z)$ is an isomorphism. So, since $G/G'$ is again torsion-free by the uniquely divisibility of $G$, the hypothesis on the dimension (or the Claim) gives $\Ext(G,Z)=\Ext(G/G',Z)=0$.
\end{proof}
	
Due to Lemma \ref{L:h_def}, given a 2-coboundary $f\in B^2(G,Z)$ we can describe a map $h:G\to Z$ with $h(1)=0$ such that $f=dh$. Nonetheless, in the torsion-free context we can describe $h$ more easily.

\begin{remark}
In the proof of the Claim, if we assume, as we may, that
		$\widehat A=A\times_f Z$ for some $2$-cocycle $f\in Z^2(A,Z)$,  then the map
		$\psi:A\to\widehat A$ is in fact $\psi(a)=n(x,0)$ where $x$ is the unique
		element of $A$ such that $nx=a$. Thus 
		\[
		\psi(a)=n(x,0) = \left( a, \sum_{i=1}^{n-1} f(x,ix) \right).
		\]
		So, setting $h:A\to Z$ to be $h(a)=\sum_{i=1}^{n-1} f(x,ix)$, we can easily see
		using that $\psi:A\to \widehat{A}$ is an homomorphism that $f=dh$ and hence
		$f\in B^2(A,Z)$. 
	\end{remark}

	\begin{cor}%\label{C:Torsion-free_Cocycle}
		Let $G$ be a torsion-free definable group in an o-minimal structure. Given a
		$2$-cocycle $f:G\times G\to Z$, define
		\[
		h: G\rightarrow Z, \ x \mapsto \sum_{k=1}^{n-1}f(y,y^k),
		\]	
		where $y$ is the unique element of $G$ such that $y^n=x$ and $n=|Z|$. Then
		$f=dh$.	
	\end{cor}
	\begin{proof} 
		Consider the group $G\times_f Z$. As every definable abelian subgroup $A$ of $G$
		is again torsion-free, the previous remark yields that $f|_{A\times A}=dh|_{A\times A}$, now writing the
		group operation of $A$ multiplicatively. In particular, if we define
		$\tilde{f}:=f-dh$, then we obtain that $\tilde{f}(x,y)=0$ for all $x,y\in G$ with
		$xy=yx$.
		
		On the other hand, Proposition \ref{P:Torsion-free} implies that $\tilde{f}$ is
		a coboundary. So, there is some function $\tilde h:G\to Z$ such that $\tilde h(1)=0$ and 
		$\tilde{f}=d\tilde h$.   
		Now, inductively on $k\ge 1$, we see that $\tilde h(x^k) = k\tilde h(x)$ for
		every $x\in G$, the case $k=1$ being evident. Indeed, assuming $\tilde h(x^k)=k\tilde
		h(x)$, we obtain that
		\[
		0=\tilde{f}(x,x^k)=d\tilde h(x,x^k) = \tilde h(x^{k+1})-\tilde h(x)-\tilde
		h(x^k) =\tilde h(x^{k+1}) - (k+1)\tilde h(x).
		\]
		So, in particular we have proven that $\tilde h(x^n) = n\tilde h(x)=0$, since
		$|Z|=n$. Hence, as $G$ is divisible, we deduce that $\tilde h=0$ and so
		$\widetilde{f}=0$, as required. 
	\end{proof}
	
	Conversano proved in \cite[Thm.\,3.6]{aC09} (cf.\,\cite[Prop.\,2.1]{CP12}) that every definably connected definable group $G$ in an arbitrary o-minimal structure contains a maximal normal definable torsion-free subgroup,
	which we denote by $N(G)$. Furthermore, she also proved in \cite[Thm.\,4.3]{aC09} (cf.\,\cite[Prop.\,2.2]{CP12})  that for a solvable group $G$ the quotient $G/N(G)$ is definably compact, so it is
	abelian by \cite[Cor.\,5.4]{PPS00}. So, for every natural number $n\ge 2$ the subgroup of $n$-torsion elements
	of $G/N(G)[n]$ is finite and isomorphic to $(\Z/n\Z)^{\dim(G/N(G))}$ by the
	structure theorem for definably compact definably connected abelian groups due to Edmundo and Otero \cite{EO04} in o-minimal expansions of real closed fields and to Edmundo et al. \cite[Thm.\,1.1]{EMPRT17} in arbitrary o-minimal structures.

	\begin{theorem}\label{T:Sol}
		Let $G$ be a  definably connected solvable $\emptyset$-definable group in an o-minimal structure $M$ with constants for the elements of the finite abelian group $Z$. The inclusion map
		$\Ext_d(G,Z) \hookrightarrow \Ext(G,Z)$ is an isomorphism and both groups are isomorphic to
		the finite group $Z^{\dim(G/N(G))}$.
	\end{theorem}
	\begin{proof}
		Let $N(G)$ be the maximal normal definable torsion-free subgroup of $G$, and let
		$\bar G=G/N(G)$ which is definably compact and abelian, as we have just
		remarked. Thus, as $\bar G$ is abelian, every abstract finite central extension of $\bar G$ is naturally interpretable by \cite[Thm.\,2.9]{BPP10}. So, the injective homomorphism  $\Ext_d(\bar G,Z)\hookrightarrow \Ext(\bar G,Z)$ given by Corollary \ref{C:Ext-Inclusion} is an isomorphism, by Lemma \ref{L:Natural_Int}.
		
		Now, as remarked above, the subgroup $N(G)$ is uniquely divisible, so it is
		definably connected. Thus, by considering the following exact sequence 
		\[
		1\longrightarrow N(G) \stackrel{\iota}{\longrightarrow} G
		\stackrel{\beta}{\longrightarrow} \bar G \longrightarrow 1,
		\]
		Corollary \ref{C:Key_Diagram}  yields the existence of the following commuting
		diagram
		\[
		\begin{tikzcd}
			0 \arrow{r}  & \Ext(\bar G,Z) \arrow{r}{\beta^*} & \Ext(G,Z)     
			\arrow{r}{\iota^*}      & \Ext(N(G),Z)\stackrel{\ref{P:Torsion-free}}{=}0  \\
			0 \arrow{r} & \Ext_d(\bar G,Z)
			\arrow{r}{\beta_d^*}\arrow[hookrightarrow]{u}{\cong} & \Ext_d(G,Z)     
			\arrow{r}{\iota_d^*}\arrow[hookrightarrow]{u}  &
			\Ext_d(N(G),Z)\stackrel{\ref{P:Torsion-free}}{=}0   
			\arrow[hookrightarrow]{u}{\cong}
		\end{tikzcd}
		\]
		Hence, we readily obtain that the homomorphisms
		$\beta^*_\diamond:\Ext_\diamond(\bar G,Z)\to \Ext_\diamond(G,Z)$ and hence
		$\Ext_d(G,Z)\hookrightarrow \Ext(G,Z)$ are isomorphisms.
		
		Finally, to compute $\Ext(\bar G,Z)$ write $Z$ as $\Z/k_1\Z\times \ldots\times
		\Z/k_r\Z$. Notice that every finite central extension of $\bar G$ is abelian
		\cite[Lem.\,2.8]{BPP10}. So, we can apply Theorem 52.2 and properties (F) and
		(D) from \cite[Sec.\,52]{lF70} to obtain\begin{align*}
			\Ext(\bar G,Z) & \cong \prod_j\Ext(\bar G,\Z/k_j\Z) \cong \prod_j\Ext(\bar
			G[k_j],\Z/k_j\Z) \\ & \cong \prod_j\Ext(\Z/k_j\Z,\Z/k_j\Z)^{\dim(\bar G)} \cong
			\prod_j (\Z/k_j\Z)^{\dim(\bar G)}\cong Z^{\dim(\bar G)},
		\end{align*} 
		where the third isomorphism follows from \cite[Thm.\,52.2]{lF70} and the
		structure theorem for definably compact definably connected abelian groups
		\cite[Thm.1.1]{EMPRT17}.
	\end{proof}

\begin{remark}\label{rmk:equivparam} As explained in Remark \ref{R:Parameters}, given a fixed subset of parameters $B$ of $M$, one may adapt Definition \ref{def:natinter} in order to introduce the notion of naturally interpretable with parameters over $B$. Notice that, as consequence of Theorem \ref{T:Sol}, we obtain the equivalence between being naturally interpretable with or without parameters, whenever the group $G$ is a definably connected solvable $\emptyset$-definable group in an arbitrary o-minimal structure (cf.\,\cite[Cor.\,1.2]{EJP11}).
\end{remark}
	
\section{On definable central extensions in o-minimal expansions of real closed fields}\label{s:Def_Cohom}

In this section we work within an o-minimal expansion $M$ of a real closed field. Henceforth, in this section $Z$ denotes a finite abelian group, which we can assume to be $\emptyset$-definable.  Notice that, by Remark \ref{R:Parameters}, naturally interpretable with or without parameters agree in this context. So, there is no need to be precise about parameters.

Let us briefly recall some topological results on o-minimal locally definable groups, which will play an essential role in the subsequent statements. 

A {\em locally definable group} $G$ of $M^k$ is an increasing countable union $G=\bigcup_{i\in \N} X_i$ of definable subsets 	$X_i$ together with a group operation $\cdot$  such that for each 	$i\in \N$ there is some $j\in \N$ with $X_i\cdot X_i\subset X_j$ and the
operation $X_i\times X_i \rightarrow X_j$ is definable. A homomorphism $f:G\to H$
between locally definable groups $G=\bigcup_{i\in \N} X_i$ and $H=\bigcup_{i\in \N} Y_i$ is locally definable if for every definable subset of the domain the restriction is a definable map, {\it i.e.} for each $n$ there is some $m$ such that $f(X_n)\subset Y_m$ and $f|_{X_n}:X_n\to Y_m$ is definable.
	
Let $G=\bigcup_{i\in \N} X_i$ be a locally definably connected definable group in an o-minimal expansion of a real closed field. The {\em o-minimal fundamental group} $\pi_1(G)$ of $G$ is defined in the usual way, but with loops and homotopies being definable. So, the group $\pi_1(G)$ is the group formed by the definable homotopy equivalence classes of continuous definable loops 	$\sigma:[0,1]\to M^k$ with $\sigma(0)=\sigma(1)=1_G$ and $\sigma([0,1])\subset X_n$ for some $n$, where as remarked above the equivalence between two loops is witnessed by a definable homotopy $[0,1]\times[0,1]\to X_m$ between them for a large enough $m$. As usual, the group operation is given by concatenation of loops. Edmundo \cite{mE05} and Edmundo and Eleftheriou \cite[Thm.\,1.4]{EE07} proved that, as in the topological setting, the o-minimal fundamental group $\pi_1(G)$ is naturally 	isomorphic to the kernel of the so-called {\em universal o-minimal covering} of $G$. The universal o-minimal covering of $G$ was first constructed in
\cite[Sec.\,3]{mE05} and it consists of a locally definable group $\widetilde{G}$ together with a locally definable surjective homomorphism $p_G:\widetilde{G}\to G$ with countable kernel such that $\pi_1(\widetilde G)$ is trivial.
	
As stated in \cite[Fact.\,6.13]{BO10loc} (cf.\,\cite[Thm.\,5.11]{DK84}) we can also describe $\widetilde{G}$ in a rather simple way, see the proof of \cite[Fact.\,3.4.13]{eB09} for a detailed explanation. Namely, consider the collection $\mathcal{P}$ of all locally definable continuous paths $\alpha:[0,1]\rightarrow G$ such that $\alpha(0)=1_G$. Let $\sim$ be the 	equivalence relation on $\mathcal{P}$ defined as follows: $\alpha \sim \beta$ if and only if $\alpha(1)=\beta(1)$ and $[\alpha*\overline\beta]=[0]$ in $\pi_1(G)$, where $*$ denotes the concatenation of maps and $\overline\beta$ is the map $x\mapsto \beta(1-x)$. Denote by $\alpha_{\sharp}$ the equivalence class of $\alpha\in\mathcal P$ and define $\widetilde{G}=\mathcal{P}/\sim$, with the group operation induced by the path product. Let then $p_G:\widetilde{G}\rightarrow G$ be defined as $p_G(\alpha_{\sharp})= \alpha(1)$ and notice that the natural map $\text{ker}(p_G)\to \pi_1(G)$, defined by $\alpha_{\sharp}\mapsto [\alpha]$, is an isomorphism. 

We summarize some key properties of the universal covering in the following facts.

\begin{fact}\label{F:Universal0}
Let $G$ and $K$ be two definably connected definable groups in an o-minimal expansion of a real closed field. If $\rho:K\to G$ is a definable homomorphism, then we have the following commutative diagram of locally definable
homomorphisms
\[
\begin{tikzcd}
	0 \arrow{r} & \pi_1(K) \arrow{r}{\cong}\arrow{d}{\rho_*} & \ker(p_{K})
	\arrow{r}\arrow{d}  & \widetilde{K}      
	\arrow{r}{p_{K}}\arrow{d}{\widetilde\rho}        & K \arrow{r}\arrow{d}{\rho}  & 1
	\\
	0 \arrow{r}  & \pi_1(G) \arrow{r}{\cong} & \ker(p_G) \arrow{r}  & \widetilde G
	\arrow{r}{p_G}       & G       \arrow{r} & 1 
\end{tikzcd}
\]
where  $\rho_*$ and $\widetilde\rho$ are defined as $\rho_*([\sigma])=[\rho\circ\sigma]$ and $\widetilde\rho(\alpha_\sharp)= (\rho\circ\alpha)_\sharp$ respectively, and both rows are central exact sequences. Furthermore, we have
\begin{enumerate}
	\item If $\rho_*$ is injective and $\rho$ has finite kernel, then $|\ker(\widetilde\rho)|\le|\ker(\rho)|$.
	\item If $\rho$ is surjective and with finite kernel, then $\widetilde\rho$ is an isomorphism.
\end{enumerate}
\end{fact}	
\begin{proof}The commutativity of the diagram is clear. Also, note that each locally definable subgroup $\ker(p_\bullet)$ is countable ({\it i.e.} has dimension $0$)  and so it is central by \cite[Cor.\,3.16]{mE06}. Now, we prove (1) and (2).
	
\noindent (1) Let $N=\ker(\widetilde \rho)$. Note that $p_{K}(N)\subset \ker(\rho)$, so $p_K(N)$ is finite. It then follows that the index $[N:N\cap \ker(p_K)]\le |\ker(\rho)|$. But by assumption $\rho_*$ and hence $\widetilde\rho|_{\ker(p_K)}$ are injective, so applying $\widetilde\rho$ we get that $N\cap \ker(p_K)$ is trivial and $N$ is finite of size at most $|\ker(\rho)|$.

\noindent (2) We deduce by considering lifting of paths that
$\widetilde \rho$ is surjective. Indeed, given a definable path $\alpha:[0,1]\to G$, by \cite[Prop.\,2.6 and Prop.\,2.11]{EO04} there is a unique definable path $\alpha':[0,1]\to K$ with $\rho\circ\alpha'=\alpha$, so $\widetilde\rho(\alpha'_\sharp) = \alpha_\sharp$. This shows that $\widetilde\rho$ is surjective. On the other hand, as $\rho$ is onto and with finite kernel, by \cite[Cor.\,2.8 and
Prop.\,2.11]{EO04} the map $\rho_*$ is injective and so $\ker(\widetilde \rho)$ is finite by (1). By \cite[Prop.\,3.4 and 3.12]{mE05} we then have that
\[
\ker(\widetilde\rho)\cong \pi_1(\widetilde G)/\widetilde\rho_*(\pi_1(\widetilde K)).
\]
Since $\pi_1(\widetilde G)$ is trivial, the map $\widetilde\rho$ is injective.
\end{proof}

%The construction of $\pi_1(G)$ and $\widetilde{G}$ can be carried out even when $G$ is locally definable. Since this generality is not required in the sequel, we have restricted ourselves to the definable context, but we refer the interested reader to \cite{BO10loc,EE07}. 
	
%We only need the following universal property:
	
	\begin{fact}\label{F:Universal}
		Let $G$ be a definably connected definable group in an o-minimal expansion of a real closed field and let $\rho:G_1\to G$ be a
		definable finite central extension with $G_1$ definably connected. Then, there is
		a locally definable onto homomorphism $p_{1}:\widetilde{G}\to G_1$ such that
		$\rho\circ p_{1} = p_G$. 
	\end{fact}
	\begin{proof} It suffices to apply Fact \ref{F:Universal0} and set $p_{1}=p_{G_1}\circ\widetilde{\rho}\inv$.
	\end{proof}

Recall that a definable group $G$ is {\em simply-connected} if it is path-connected and $\pi_1(G)$ is trivial.
	
	\begin{lemma}\label{L:Simply-Con}
	Let $G$ be a simply-connected group definable in an o-minimal expansion of a real closed field. Then $\Ext_d(G,Z)$ is trivial.
	\end{lemma}
\begin{proof}
Note first that $G$ is definably connected, since it is path-connected, and that $p_G:\widetilde{G}\to G$ is an isomorphism, since $\pi_1(G)$ is trivial. We may assume that $\widetilde{G}=G$.

It suffices to prove that every definable finite central extension 
\[
0\longrightarrow Z\lhook\joinrel\longrightarrow G_1\stackrel{\rho}{\longrightarrow} G\longrightarrow 1
\] splits trivially. By Fact \ref{F:Universal} there is some definable onto homomorphism $p_1: G\to G_1^\circ$ such that $\rho\circ p_1=\mathrm{id}_G$.
%Indeed, the map \[\alpha: G\to Z\cap G_1^\circ, \ x\mapsto x\inv\rho(p_1(x))\]is an homomorphism, since $\im(\alpha)$ is central in $G_1$ because $p_1\circ \rho=\mathrm{id}_G$.
 This implies that %$\alpha$ is trivial, by definably connectedness of $G$ (and Fact \ref{F:Connected}), yielding that 
 the exact sequence 
\[
0\longrightarrow Z\cap G_1^\circ \lhook\joinrel\longrightarrow G_1^\circ \longrightarrow G\longrightarrow 1
\]
 splits trivially. However, this forces that $Z\cap G_1^\circ$ is the trivial subgroup and consequently $G_1=G_1^\circ\times Z\cong G\times Z$, as desired. 
\end{proof}

Recall that a definable subset $X\subset M^k$ is called {\em definably contractible} if there is an element $x_0\in X$ and a definably homotopy $H:X\times[0,1]\to X$ between the identity map $\mathrm{id}_X$ on $X$ and the constant map $x\mapsto x_0$. In particular, a definably contractible definable group $G$ satisfies that $\pi_1(G)=\{0\}$.

%\begin{fact}{\cite[Prop.\,2.5]{aC14}}
%A definably connected group is definably contractible if and only if it is torsion-free.
%\end{fact}

\begin{prop}\label{P:Def_Contractible}Let $G$ be a definably connected definable group in an o-minimal expansion of a real closed field. Let $K$ and $L$ be definably connected definable subgroups of $G$, such that $L$ is definably contractible, $K\cap L=\{1\}$ 	and $G=K\cdot L$. Then, the homomorphism
\[
\Ext_d(G,Z) \to \Ext_d(K,Z), \ [f] \mapsto [f|_{K\times K}]
\]
is surjective. 
\end{prop}
\begin{proof}
Note first that the definable map $K\times L\to G$ given by $(x,y)\mapsto xy$ is a homeomorphism, so
\[
\pi_1(K) \cong \pi_1(K)\times \pi_1(L) \cong \pi_1(K\times L) \cong \pi_1(G).
\]
In fact, if $\iota:K\hookrightarrow G$ denotes the inclusion map, then the above isomorphism is obtained via the map $\iota_*: \pi_1(K)\to \pi_1(G)$. Thus $\widetilde{\iota}:\widetilde{K}\to\widetilde{G}$ is injective by Fact \ref{F:Universal0}.

Now, given $g\in Z^2_d(K,Z)$, let $\widehat K=K\times_g Z$ and $Z_1=Z\cap \widehat{K}^\circ$. Note that 
\[
0\longrightarrow Z_1\lhook\joinrel\longrightarrow \widehat{K}^\circ \stackrel{\pi_0}{\longrightarrow} K \longrightarrow 1
\] 
is a finite central extension and that a definable section $s:K\rightarrow \widehat{K}^\circ$ is also a section of $s:K\rightarrow \widehat{K}$.  So, the definable $2$-cocycle $g_s$ determined by $s$ is equivalent to $g$. Thus, we may assume that $\im(g)\subset Z_1$ and  $\widehat K^\circ=K\times_g Z_1$.
		
By the universal property of $p_K:\widetilde K\to K$ (Fact \ref{F:Universal}), there is a locally definable surjective homomorphism $\gamma:\widetilde{K}\to \widehat{K}^\circ$ with $\pi_0\circ \gamma=p_K$. We then have the following commutative diagram:
\[
\begin{tikzcd}
		0 \arrow{r}  & \ker(p_G) \arrow{r}  & \widetilde G
	\arrow{r}{p_G}       & G       \arrow{r} & 1  \\
	0 \arrow{r}  & \ker(p_K) \arrow{r}\arrow{d}\arrow[hook]{u}{\cong} & \widetilde K      
	\arrow{r}{p_K}\arrow{d}{\gamma}  \arrow[hook]{u}{\widetilde \iota}      & K \arrow{r}\arrow[equal]{d}\arrow[hook]{u}{\iota}  & 1 \\
	0 \arrow{r} & Z_1 \arrow{r}  & \widehat{K}^\circ \arrow{r}{\pi_0}       & K    
	\arrow{r} & 1 
\end{tikzcd}
\]
Let $\Gamma=\text{ker}(\gamma)$, so $\Gamma\subset \ker(p_K)$ and $\widetilde{K}/\Gamma$ is isomorphic to $\widehat{K}^\circ$, which naturally induces an isomorphism between $\text{ker}(p_K)/\Gamma$ and $Z_1$, by the left column. Finally, since $\widetilde{\iota}(\Gamma)$ is contained in the central subgroup $\ker(p_G)$ of $\widetilde{G}$, we can consider the locally definable subgroup $G_1^\circ:=\widetilde{G}/\widetilde{\iota}(\Gamma)$ and the surjective group homomorphism
\[
\pi:G_1^\circ\rightarrow G,\ [x]\mapsto p_G(x).
\]
The kernel of $\pi$ is $\ker(p_G)/\widetilde{\iota}(\Gamma)=\widetilde{\iota}(\ker(p_K))/\widetilde{\iota}(\Gamma)$, which is isomorphic to $\ker(p_K)/\Gamma$ and hence to $Z_1$. As $G_1^\circ$ is locally definable and $Z_1$ is finite, we deduce that $G_1^\circ$ is definable. We have obtained the following commutative diagram
\[
\begin{tikzcd}
	0 \arrow{r}  & \ker(\pi) \arrow{r} & G_1^\circ \arrow{r}{\pi} & G \arrow{r}  & 1 \\
	0 \arrow{r} & Z_1 \arrow{r}\arrow[hook]{u}{\cong} & \widehat{K}^\circ \arrow{r}{\pi_0}\arrow[hook]{u}{\iota_1}      & K       \arrow{r}\arrow[hook]{u} & 1 
\end{tikzcd}
\]
where both rows are exact. Moreover, the definable section $s_0:K\to\widehat{K}^\circ$ defined by $x\mapsto (x,0)$ can be extended to a definable section $s:G\to G_1^\circ$ satisfying $s|_K = \iota_1\circ s_0$. Thus, regarding the $2$-cocycle $f\in Z_d^2(G,Z_1)$ determined by $s$  as a $2$-cocycle in $Z_d^2(G,Z)$, we see that $[f|_{K\times K}]=[g]$ and in fact note that $G_1^\circ$ corresponds to the connected component of $G\times_f Z$.
\end{proof}
	
While in general an o-minimal group $G$ might not possess a maximal definably compact subgroup (see \cite[Ex.\,5.3]{aS94}), Conversano proved \cite[Thm.\,1.2]{aC14}:

\begin{fact}\label{fact:conversano}
Let $G$ be a definably connected definable group in an o-minimal structure. Set $\bar G := G/N(G)$. Then $\bar G$ has a maximal definably compact subgroup $\bar K$, which is definably connected and it is unique up to conjugation. Moreover, there is a definable torsion-free subgroup $\bar L$ in $\bar G$ such that $\bar G=\bar K\cdot \bar L$ and $\bar K\cap \bar L=\{\bar 1\}$.

\end{fact}

Next, we see that $\Ext_d(\bar K,Z)$ determines $\Ext_d(G,Z)$.
	
\begin{cor}Let $G$ be a definably connected definable group in an o-minimal expansion of a real closed field. Let $\bar K:=K/N(G)$ be a maximal definably (connected) compact definable subgroup of $\bar G:=G/N(G)$ and let $\pi:G\to\bar G$ be the natural projection. Then, the induced map $ \pi_d^*: H_d^2(\bar G,Z) \longrightarrow H_d^2(G,Z)$ is an isomorphism. In particular, the map
\[
\Ext_d(G,Z) \to \Ext_d(\bar K,Z), \ [f] \mapsto [\bar f|_{\bar K\times \bar K}],
\]
for some $\bar f\in Z_d^2(\bar G,Z)$ satisfying $[f]=[\bar f\circ(\pi\times\pi)]$, is an isomorphism. 
\end{cor}
\begin{proof}
Theorem \ref{T:Ext-Sequence} applied to the exact sequence $1\to N(G)\to G\stackrel{\pi}{\to} \bar G\to 1$ yields
\[
0\longrightarrow H_d^2(\bar G,Z) \stackrel{\pi_d^*}{\longrightarrow} H_d^2(G,Z) \longrightarrow H_d^2(N(G),Z)\stackrel{\ref{P:Torsion-free}}{=}0.
\]
So, the map $\pi_d^*$ is an isomorphism. For the second part of the statement, it suffices to prove that the map $\Ext_d(\bar G,Z) \to \Ext_d(\bar K,Z)$, defined by $[\bar f] \mapsto [\bar f|_{\bar K\times \bar K}]$, is an isomorphism. 

In an o-minimal expansion $M$ of a real closed field, a definable torsion-free group is definably homeomorphic to $M^k$ by \cite[Cor.\,5.7]{PS05} and so it is definably contractible. Thus, Proposition \ref{P:Def_Contractible} together with the previous Fact \ref{fact:conversano}  yields that the map is surjective. Hence, it only remains to see that it is injective. To ease notation we may assume that $N(G)$ is trivial, so $\bar G=G$ and $\bar K=K$.

Let $f\in Z_d^2(G,Z)$ be such that $[f|_{K\times K}]=[0]$. In particular, there is some definable $h_K:K\to Z$ with $h_K(1)=0$ such that $f|_{K\times K}=dh_K$. Denote by $r_K:G\to K$ the map defined by $r_K(x)=x$ for $x\in K$ and $r_K(x)=1$ otherwise. 
So, after replacing $f$ by $f-dh$ where $h=h_K\circ r_K$, we may further assume that $f|_{K\times K}=0$. 

Set now $\widehat{G}=G\times_f Z$ and $Z_1=Z\cap \widehat{G}^\circ$. Note that  $K\times_f Z=K\times Z$, so $K\times\{0\}$ is a definably connected subgroup of $\widehat G^\circ$. Thus, we can choose a definable section $s:G\to \widehat{G}^\circ$ such that $s(x)=(x,0)$ for every $x\in K$. Since $s$ is 		also a section $s:G\to \widehat{G}$, the definable $2$-cocycle $f_s$
determined by $s$ is equivalent to $f$. Moreover, note that $f_s|_{K\times K}=0$.
So, after replacing $f$ by $f_s$, we may assume that $\im(f)\subset Z_1$ and $\widehat
		G^\circ=G\times_f Z_1$ with $f|_{K\times K}=0$.
		
Consider the associated central extension 
		\[
		0\longrightarrow Z_1 \longrightarrow \widehat G^\circ
		\stackrel{\pi}{\longrightarrow} G \longrightarrow 1,
		\] 
		which is definable by assumption. We claim that $\pi\inv (K)=K\times Z_1$ is a maximal definably
(connected) compact subgroup of $\widehat G^\circ$. For otherwise, let $\widehat{K}_1$ be
a definably compact subgroup of $\widehat{G}$ containing $\pi\inv(K)$.
So, we have $K=\pi(\pi\inv (K))\subset\pi(\widehat K_1)$ and therefore $K=\pi(\widehat{K}_1)$, as $\pi(\widehat K_1)$ is definably (connected) compact subgroup of $G$. Thus $\pi\inv(K)=\widehat{K}_1$ by maximality. Since $\pi\inv(K)=K\times Z_1$, we deduce that $Z_1$ is trivial and
therefore $\widehat{G}\cong G\times Z$, as required.		\end{proof}

	\begin{prop}\label{prop:CompactvsAbelianDef}Let $K$ be a definably connected
		definably compact definable group in an o-minimal expansion of a real closed field. Let $A$ be an
		abelian definable subgroup of $K$ of maximal dimension among dimensions of abelian definable subgroups of $K$. Then, the homo\-morphism
		\[
		\Ext_d(K,Z) \to \Ext_d(A,Z), \ [f] \mapsto [f|_{A\times A}]
		\]
		is injective.
	\end{prop}

	\begin{proof} As in the proof of the previous proposition, given $f\in Z_d(K,Z)$ with $[f|_{A\times A}]=[0]$ and setting $\widehat{K}=K\times_f Z$ and $Z_1=Z\cap
		\widehat{K}^\circ$, we may assume after some reductions that $\widehat K^\circ=K\times_f Z_1$ and that we have a central extension $0\rightarrow Z_1 \rightarrow \widehat K^\circ
		\stackrel{\pi}{\rightarrow} A \rightarrow 1$,
		which is definable and $\pi\inv(A)=A\times Z_1$. Moreover, we also have that $A$ is definably connected and a maximal abelian subgroup, by \cite[Fact\,37 and Lem.\,38]{BJO14}.
		
		Now, if $A_1$ is an abelian definable subgroup of $\widehat K^\circ$ containing $\pi\inv(A)$ then $\pi(A_1)$ is clearly an abelian definable group containing $A$, so $\pi(A_1)=A$ by maximality. Therefore, we deduce that $A_1=\pi\inv(A)$ and $\pi\inv(A)$ is a maximal abelian definable subgroup of $\widehat K^\circ$. Therefore, since $\pi\inv(A)$ has maximal dimension among dimensions of abelian definable subgroups of $K$,  it is definably connected by \cite[Fact\,37 and Lem.\,38]{BJO14}. This implies that $Z_1$ is the trivial subgroup, and so $\widehat K=\widehat K^\circ\times Z \cong K \times Z$, which yields the result. 
	\end{proof}

Altogether, the last two results and Theorem \ref{T:Sol} yield:	

\begin{cor}\label{C:FiniteExt}
Let $G$ be a definably connected definable group in an o-minimal expansion of a real closed field. Then, the group $\Ext_d(G,Z)$ embeds into $Z^k$, where $k$ is bounded above by the dimension of $G/N(G)$. In particular, there are finitely many non equivalent definable central extensions of $G$ by $Z$.\qed
\end{cor}
	
\section{A reduction to simple groups}\label{s:Reduction}

In this final section we prove that in the context of o-minimal expansions of real closed fields, the conjecture holds for arbitrary definable groups if it does for (almost) definably simple groups. Recall that all finite abelian  groups can be assumed to be $\emptyset$-definable. 

The following lemmata will be necessary to prove this reduction.
	
\begin{lemma}\label{L:UpDown} Let $0\to  F\to G_0 \to  G\to 1$ be an exact sequence of groups definable in an o-minimal expansion of a real closed field, where $G_0$ and $G$ are definably connected definable groups, and $F$ is a finite central subgroup of $G_0$.
\begin{enumerate}
\item If $Z$ is a finite abelian group and $\Ext_d(G_0,Z)\hookrightarrow \Ext(G_0,Z)$ is an isomorphism, then $\Ext_d(G,Z)\hookrightarrow \Ext(G,Z)$ is an isomorphism. 
			
\item If $\Ext_d(G,Z)\hookrightarrow \Ext(G,Z)$ is an isomorphism for every finite abelian group $Z$, then $\Ext_d(G_0,Z)\hookrightarrow \Ext(G_0,Z)$ is also an isomorphism for every finite abelian group $Z$.	
			
\end{enumerate}		
\end{lemma}
\begin{proof}
Let $\beta:G_0\rightarrow G$ denote the homomorphism of the statement and assume, as we may, that its kernel is $F$.
		
\noindent (1) Let $f:G\times G\rightarrow Z$ be an arbitrary $2$-cocycle. We need to prove that $f$ is equivalent to a definable $2$-cocycle, that is, the group extension given by $\widehat{G}:=G\times_{f} Z$ is naturally interpretable (see Lemma \ref{L:Natural_Int}). 

By Lemma \ref{L:Inf-Ret} we have a commuting diagram
		\[
		\begin{tikzcd}
			0 \arrow{r}  & \mathrm{Hom}(F,Z) \arrow{r} & \Ext(G,Z)      \arrow{r}{\beta^*}    &
			\Ext(G_0,Z) \\
			0 \arrow{r} & \mathrm{Hom}(F,Z) \arrow{r} \arrow[hookrightarrow]{u}{\cong} &
			\Ext_d(G,Z)   \arrow{r}{\beta_d^*} \arrow[hookrightarrow]{u}  &
			\Ext_d(G_0,Z)
			\arrow[hookrightarrow]{u}{\cong}
		\end{tikzcd}
		\]	
where both rows are exact. By assumption, there exists some definable $2$-cocycle $g:G_0\times G_0\rightarrow Z$ such that $[g]=\beta^*([f])=[f\circ (\beta \times \beta)]$, so $f\circ(\beta\times \beta)-g=dh$ for some $h:G_0\to Z$ with $h(1)=0$. Consider the group $\widehat{G_0}:=G_0\times_{g} Z$. The canonical  homomorphism $\gamma:\widehat{G_0}\to \widehat G$, given by $(x,z)\mapsto (\beta(x),z+h(x))$, induces an exact sequence
\[
0\longrightarrow \Delta_1 \longrightarrow \widehat{G_0} \longrightarrow \widehat{G} \longrightarrow 1
\]
where $\Delta_1=\left\{ (x,z)\in F\times Z\, :\, h(x)=-z \right\}$. 
As $\widehat{G_0}$ is definable and $\Delta_1$ is finite, we deduce that $ \widehat{G_0}/\Delta_1 $ is definable as well, and isomorphic to $\widehat{G}$. This yields that 
\[
\begin{tikzcd}
	0 \arrow{r}  & Z \arrow{r}{\mu} \arrow[equal]{d} & \widehat{G_0}/\Delta_1      \arrow{r}{\pi} \arrow{d}{\bar \gamma}   &
	G \arrow{r} 	\arrow[equal]{d}  & 1 \\
	0 \arrow{r} & Z \arrow{r}  &
	\widehat{G}   \arrow{r}   &
	G \arrow{r} & 1
\end{tikzcd}
\]	
is a commutative diagram, with $\mu(z)=[(1,z)]$, $\pi([(x,z)])=\beta(x)$ and $\bar\gamma$ the isomorphism map induced by $\gamma$. Hence, by Lemma \ref{L:Natural_Int} the cocycle $f$ is equivalent to a definable one, as required.
		
\noindent (2) Let $Z$ be a finite abelian group, and consider an arbitrary abstract finite central extension
\[
0\longrightarrow Z \longrightarrow \widehat{G_0} \stackrel{\pi}{\longrightarrow} G_0 \longrightarrow 1.
\]  
We prove that this extension is naturally interpretable.

\begin{claim*}
We can assume that $\widehat{G_0}$ does not have subgroups of finite index. 
\end{claim*}
\begin{claimproof*} Let $N$ be a finite index normal subgroup of $\widehat{G_0}$ without subgroups of finite index, which exists by Fact \ref{F:Connected}. Let $Z_0:=Z\cap N$, and note that a section $s$ of $G_0\to N$ is also a section of $G_0\rightarrow \widehat{G_0}$. Thus, for the $2$-cocycle $f_s:G_0\times G_0\rightarrow Z_0$ determined by $s$ we have that $\widehat{G_0}$ is equivalent to $G_0\times_{f_s} Z$. Hence, it is enough to show that $f_s$ is equivalent to a definable $2$-cocycle.
\end{claimproof*}	
		
Since $\pi^{-1}(F)$ is normal in $\widehat{G_0}$, we have that $\widehat{G_0}$ acts on $\pi^{-1}(F)$ by conjugation. Since $\widehat{G_0}$ has no proper finite-index subgroup by the Claim, this yields that $\pi^{-1}(F)$ is central in $\widehat{G_0}$. Therefore, the following is also a finite central exact sequence
\[
0\longrightarrow \pi^{-1}(F) \longrightarrow \widehat{G_0} \stackrel{ \beta\circ\pi}{\longrightarrow} G \rightarrow 1.
\]  
Since $\Ext_d(G,\pi^{-1}(F))\hookrightarrow \Ext(G,\pi^{-1}(F))$ is an isomorphism by assumption, we can assume that $\widehat{G_0}$ and $p:=\beta\circ \pi$ are definable. Hence, it remains to prove that $\pi$ is definable. 
	
Consider the definable subgroup 
\[
\widehat{G_0}\times_{p,\beta} G_0:=\left\{(x,y)\in \widehat{G_0}\times G_0 \, : \,p(x)=\beta(y)\right\}
\]
of $\widehat{G_0}\times G_0$ and its subgroup 
\[
\Delta_2:=\left\{(x,\pi(x)) \, :\,  x\in \widehat{G_0}\right\}.
\]
Note that the map $\widehat{G_0}\rightarrow \Delta_2$ given by $x\mapsto (x,\pi(x))$ is an isomorphism and therefore $\Delta_2$ does not have subgroups of finite index, as neither does $\widehat{G_0}$. Thus, $\Delta_2$ is contained in $(\widehat{G_0}\times_{p,\beta} G_0)^\circ$. Moreover, it has finite index in $\widehat{G_0}\times_{p,\beta} G_0$. Indeed, for $(x,y)\in \widehat{G_0}\times_{p,\beta} G_0$ we have that $\pi(x)y\inv\in F$ and $F$ is finite, so the subgroup $\Delta_2$ has at most $|F|$ many cosets. It then follows that $\Delta_2 = (\widehat{G_0}\times_{p,\beta} G_0)^\circ$, so $\Delta_2$ is definable and so is $\pi$. 
\end{proof}

\begin{lemma}\label{L:AlmostProd}
Let $G$ be a definably connected definable group in an o-minimal expansion of a real closed field and let $Z$ be a finite abelian group. Suppose that $G$ is the almost direct product of two definably connected normal subgroups $A$ and $B$, {\it i.e.} $G=A\cdot B$ and $A\cap B$ is finite. Consider the homomorphism
		\[
		\gamma:\Ext(G,Z) \to \Ext(A,Z)\times \Ext(B,Z), \ [f]\mapsto \big([f|_{A\times
			A}] , [f|_{B\times B}]\big)
		\]
		and let $f\in Z^2(G,Z)$. If $\gamma([f])=([f_1],[f_2])$ for some $f_1\in Z_d^2(A,Z)$ and some $f_2\in 		Z_d^2(B,Z)$, then $[f]\in \Ext_d(G,Z)$.
	\end{lemma}
	\begin{proof}	
		Let $f\in Z^2(G,Z)$ be such that $\gamma([f])=([f_1],[f_2])$ for some $f_1\in
		Z_d^2(A,Z)$ and some $f_2\in Z_d^2(B,Z)$. So, there are some  $h_A:A\to Z$ and
		$h_B:B\to Z$ with $h_A(1)=h_B(1)=0$ such that $f|_{A\times A}=f_1+dh_A$ and $f|_{B\times B}=f_2+dh_B$. 
		
		\begin{claim}
		We may assume that $h_A|_{A\cap B} = h_B|_{A\cap B}$.
		\end{claim}
		\begin{claimproof}
		Let $C=A\cap B$ and let $\tilde h_A:A\to Z$ be the function 
		\[
		\tilde h_A(x)=(h_B\circ r_C)(x) + (h_A\circ r_{A\setminus C})(x),
		\]
		where for a subset $X\subset G$ the map $r_X:G\to X\cup\{1\}$ is defined as the identity on $X$ and $r_X(x)=1$ for $x\notin X$. 
		It is clear that $\tilde h_A$ and $h_B$ agree on $C$. Replacing $f_1$ by $\tilde{f}_1=f|_{A\times A}  - d\tilde{h}_A$, we obtain the claim. So, it suffices to prove that $\tilde f_1$ is definable. 
		
		Note first that $\tilde{f}_1|_{(A\setminus C)\times (A\setminus C)}= f_1|_{(A\setminus C)\times (A\setminus C)}$. In particular, we have that  $\tilde f_1$ is definable on $A\setminus C$. Moreover, since $C$ is finite, the function $\tilde{f}_1|_{C\times C}$ is also definable. Hence, it remains to check that both $\tilde f_1|_{C\times (A\setminus C)}$ and $\tilde f_1|_{(A\setminus C)\times C}$ are definable.
		
		Let $c\in C$ and $x\in A\setminus C$. An easy calculation yields that
		\[
		\tilde f_1(c,x) = f(c,x) - h_A(cx) + h_B(c) + h_A(x) = f_1(c,x) - h_A(c) + h_B(c).
		\]
		Since $C$ is finite, the maps $h_A|_C$ and $h_B|_C$ are definable (over $C$) and thus, the map $\tilde f_1|_{C\times (A\setminus C)}$ is definable, as so is $f_1|_{C\times (A\setminus C)}$. Likewise, one verifies that $\tilde f_1|_{(A\setminus C)\times C}$ is definable. It then follows that $\tilde f_1$ is definable. 
		\end{claimproof}
	
	Once we know that $h_A|_{A\cap B} = h_B|_{A\cap B}$, we claim:
	
	\begin{claim}
			We may assume that $f|_{A\times A}=f_1$ and $f|_{B\times B}=f_2$.
	\end{claim}
\begin{claimproof}
	Using the notation from Claim 1, it suffices to define the following maps from $G$
	to $Z$:
	\[
	h_1(x) = (h_A\circ r_A)(x), \ h_2(x) =(h_B\circ r_B)(x) \, 
	\text{ and }\, h_3(x) =(h_B\circ r_{A\cap B})(x).
	\]
	After replacing $f$ by the equivalent $2$-cocycle $f-dh_1-dh_2+dh_3$, we may
	assume that $f|_{A\times A}=f_1$ and $f|_{B\times B}=f_2$. 
\end{claimproof}
		
Consider the exact sequence $0\to Z\to G\times_f Z \stackrel{\pi}{\to}G\to 1$. Define $A_1=A\times_{f} Z$ and $B_1=B\times_{f} Z$, and note that both $A_1$ and $B_1$ are subgroups of $G\times_f Z$ and definable. Moreover, it is clear that $G\times_f Z = A_1\cdot B_1$ and $A_1 \cap  B_1=\pi\inv(A\cap B)$ is finite. 
		
We claim that $A_1^\circ$ and $B_1$ commute, which follows from the fact that $A$ and $B$ commute. Indeed, let $a\in A_1^\circ$ and $b\in B_1$ be given. By normality, we have $\pi([a,b])\in A\cap B$ and, as $A\cap B$ is a finite normal subgroup of $G$ and $G$ is definably connected, the intersection $A\cap B$ is a central subgroup of $G$.
This  implies that the map $A_1^\circ\to A\cap B$ given by $x\mapsto \pi([x,b])$ is an homomorphism, so it must be the trivial homomorphism since $A_1^\circ$ is definably connected. Thus, the homomorphism $A_1^\circ\to Z$ given by $x\mapsto [x,b]$ is well-defined and it is again the trivial homomorphism, so $[a,b]=1$ as desired.
		
It then follows that the map $A_1^\circ \times B_1\to  G\times_f Z$ given by $(x,y)\mapsto xy$ is a surjective homomorphism whose kernel is the finite group
\[
\Delta:=\left\{(x,x\inv)\in (G\times_f Z)\times (G\times_f Z) \ | \ x \in A_1^\circ\cap B_1\right\}.
\]
This induces an isomorphism $\gamma: (A_1^\circ \times B_1)/\Delta \to G\times_f Z$ given by $\gamma([(x,y)]) = xy$. Therefore, we obtain the following exact sequence 
\[
0\longrightarrow Z \stackrel{\mu_1}{\longrightarrow} (A_1^\circ \times B_1)/\Delta \stackrel{\pi_1}{\longrightarrow} G \longrightarrow 1,
\]
where the homomorphisms $\mu_1:Z\rightarrow (A_1^\circ \times B_1)/\Delta$ and $\pi_1:(A_1^\circ \times B_1)/\Delta\to G$ are $z\mapsto [((1,0),(1,z))]$ and $[(x,y)]\mapsto \pi(xy)$ respectively. This exact sequence is clearly equivalent  via $\gamma$ to the one given by $G\times_f Z$. Therefore, we obtain the second part of the statement by Lemma \ref{L:Natural_Int}, as $(A_1^\circ \times B_1)/\Delta$ is definable. 
\end{proof}
	
In the next result we see that it suffices to prove the conjecture for definably almost simple definable factors. By a {\em definable factor} $K$ of $G$ we mean that $K=N_1/N_2$ where both $N_2\unlhd N_1$ are definable subgroups of $G$. Also, we say that $K$ is {\em definably almost simple} if $K/Z(K)$ has no proper definable normal subgroup. 

We first prove a reduction of the conjecture to the case of definably almost simple definable groups. It is worth noticing that in the statement the finite group $Z$ is fixed.

\begin{theorem}\label{T:Reduction}
Let $G$ be a definably connected group definable in an o-minimal expansion of a real closed field and let $Z$ be a finite abelian group. Suppose the following: 
\begin{enumerate}
	\item[$(\ast)_G$] for every definably almost simple definable factor $K$ of $G$, the inclusion map $\Ext_d(K,Z)\hookrightarrow \Ext(K,Z)$ is an isomorphism. 
\end{enumerate} 
Then $\Ext_d(G,Z)\hookrightarrow \Ext(G,Z)$ is an isomorphism.
\end{theorem}
\begin{proof}
Assume, as we may, that $G$ is infinite and proceed by induction on the
		dimension of the group. For $\dim(G)=1$, the group $G$ is abelian by
		\cite[Cor.\,2.15]{aP88} and therefore the result follows by
		\cite[Thm.\,2.9]{BPP10}. We consider the general case and observe that we may
		assume that $G$ is not solvable by Theorem \ref{T:Sol}.
		
		\begin{claim}\label{Clm:T1}
			If $K$ is a semisimple definable factor of $G$, then
			$\Ext_d(K,Z)\hookrightarrow \Ext(K,Z)$ is an isomorphism. 
		\end{claim}
		\begin{claimproof}
			By \cite[Thm.\,2.8]{PPS00} we know that there are some definably simple
			definable groups $\bar H_1,\ldots,\bar H_m$ such that $K/Z(K)$ is the direct
			product $\bar H_1\times \ldots\times\bar H_m$. Let $H_i$ be the normal subgroup
			of $K$ containing $Z(K)$ such that $H_i/Z(K)$ equals $\bar H_i$. So, we obtain
			that $K$ is the almost direct product of $H_1,\ldots,H_m$. Hence, since $(\ast)_G$
			implies that $\Ext_d(H_i,Z)\hookrightarrow \Ext(H_i,Z)$ is an isomorphism,
			applying Lemma \ref{L:AlmostProd} we get that  $\Ext_d(K,Z)\hookrightarrow
			\Ext(K,Z)$ is an isomorphism as well. 
		\end{claimproof}
		
		Assume first that $G$ is definably compact. In this case, by
		\cite[Cor.\,6.4]{HPP11} the derived subgroup $[G,G]$ of $G$ is definable,
		definably connected and semisimple and moreover $G$ is the almost direct product
		of $Z(G)^\circ$ and $[G,G]$, that is, $G=Z(G)^\circ\cdot [G,G]$ with
		$Z(G)^\circ\cap[G,G]$ finite. 
		Also, by Claim \ref{Clm:T1} we may suppose that $Z(G)$ is infinite, as otherwise
		$G=[G,G]$ and so we obtain the statement. Note in addition that $(\ast)_{[G,G]}$ holds, since every definable factor of $[G,G]$ is a definable factor of $G$. Therefore, by the inductive
		hypothesis, together with Lemma \ref{L:AlmostProd}, we obtain the result.
		We are thus left with the non definably compact case. 
		
		Now, suppose that $G$ is arbitrary but not definably compact and consider $N(G)$, the maximal normal definable torsion-free subgroup of $G$.
		\begin{claim}\label{Clm:T2}
			We may assume that  $N(G)$ is trivial.
		\end{claim}
		\begin{claimproof}
			Otherwise, we have that $N(G)$ is infinite, since it is definably connected, and thus $\dim(G/N(G))<\dim(G)$. Clearly, every definable factor of $G/N(G)$ is a definable factor of $G$, so condition $(\ast)_{G/N(G)}$ holds. Hence, Corollary \ref{C:Key_Diagram}, together
			with the inductive hypothesis, yields the existence of the following commutative
			diagram
			\[
			\begin{tikzcd}
				0 \arrow{r}  & \Ext(G/N(G),Z) \arrow{r} & \Ext(G,Z)      \arrow{r}     &
				\Ext(N(G),Z)\stackrel{\ref{P:Torsion-free}}{=}0 \\
				0 \arrow{r} & \Ext_d(G/N(G),Z) \arrow{r} \arrow[hookrightarrow]{u}{\cong} &
				\Ext_d(G,Z)      \arrow{r} \arrow[hookrightarrow]{u}  &
				\Ext_d(N(G),Z)\stackrel{\ref{P:Torsion-free}}{=}0 
				\arrow[hookrightarrow]{u}{\cong}
			\end{tikzcd}
			\]
			where both rows are exact. Thus, the inclusion $\Ext_d(G,Z)\hookrightarrow \Ext(G,Z)$ is an
			isomorphism and hence every finite central extension of $G$ is naturally
			interpretable. 
		\end{claimproof}
		
		Observe that, by Claim \ref{Clm:T1}, we may assume that $G$ is not semisimple.
		So, there is an infinite definably connected definable normal abelian subgroup
		$A$ of $G$, which is necessarily definably compact, by Claim \ref{Clm:T2} and
		\cite[Prop.\,2.2]{CP12}. 
		Consider the exact sequence $1\to A\stackrel{\iota}{\to} G\to G/A\to 1$. As $(\ast)_{G/A}$ holds, the
		inductive hypothesis and Corollary \ref{C:Key_Diagram} yield the following
		commutative diagram with exact rows:
		\[
		\begin{tikzcd}
			0 \arrow{r}  & \Ext(G/A,Z) \arrow{r} & \Ext(G,Z)      \arrow{r}{\iota^*}     &
			\Ext(A,Z) \\
			0 \arrow{r} & \Ext_d(G/A),Z) \arrow{r} \arrow[hookrightarrow]{u}{\cong} &
			\Ext_d(G,Z)      \arrow{r}{\iota^*_d} \arrow[hookrightarrow]{u}  & \Ext_d(A,Z)
			\arrow[hookrightarrow]{u}{\cong}
		\end{tikzcd}
		\]
		So, to prove that $\Ext_d(G,Z)\hookrightarrow \Ext(G,Z)$ is an isomorphism it
		remains to show that the map $\im(\iota^*_d)\to \im(\iota^*)$ is surjective.
		So, let $[g]\in \im(\iota^*)$ be arbitrary and let $f\in Z^2(G,Z)$ be such that
		$[f|_{A\times A}]=[g]$. 
		
Due to Claim \ref{Clm:T2}, by \cite[Thm.\,1.2]{aC14} there is a maximal definably compact non-trivial subgroup $K$ of $G$, which is definably connected, and a definable torsion-free non-trivial subgroup $L$ of $G$ such that $G=K\cdot L$ and $K\cap L=\{1\}$. Note that the maximality of $K$ forces $A\le K$, since $AK$ is a definably compact definable subgroup. Furthermore, the torsion-free subgroup $L$ of $G$ is definably contractible, by \cite[Cor.\,5.7]{PS05}. 
		
Note now that $(\ast)_K$ is trivially satisfied. So, the induction hypothesis implies that $\Ext_d(K,Z)\hookrightarrow \Ext(K,Z)$ is an isomorphism. Thus, we can apply Proposition \ref{P:Def_Contractible}
		 which yields the existence of some definable $f'\in Z^2_d(G,Z)$ such that
		$[f'|_{K\times K}] = [f|_{K\times K}]$. So, we obtain that $[f'|_{A\times A}] =
		[f|_{A\times A}]$  which implies  $\iota_d^*([f'])=[g]$. This shows that
		$\im(\iota^*_d)\to \im(\iota^*)$ is surjective, finishing the proof.
	\end{proof}

As a consequence, we also obtain a reduction of the conjecture to definably simple definable groups. Note that in this case we need the group $Z$ to vary.

\begin{cor}\label{C:Reduction}
Let $G$ be a definably connected group definable in an o-minimal expansion of a real closed field. Suppose the following: 
\begin{enumerate}
	\item[$(\star)_G$] for every definably simple definable factor $S$ of $G$, the natural inclusion map $\Ext_d(S,Z)\hookrightarrow \Ext(S,Z)$ is an isomorphism for every finite abelian group $Z$. 
\end{enumerate} 
Then $\Ext_d(G,Z)\hookrightarrow \Ext(G,Z)$ is an isomorphism for every finite abelian group $Z$. 
\end{cor}
\begin{proof}
Assume $(\star)_G$ and let $K$ be a definably almost simple factor of $G$. Lemma \ref{L:UpDown} yields that $\Ext_d(K,Z)\hookrightarrow \Ext(K,Z)$ is an isomorphism for every finite abelian group $Z$. So, condition $(\ast)_G$ of Theorem \ref{T:Reduction} is satisfied for a fixed $Z$. Hence, Theorem \ref{T:Reduction} implies that $\Ext_d(G,Z)\hookrightarrow \Ext(G,Z)$ is an isomorphism for every finite abelian group $Z$. 
\end{proof}

We conclude with the following reduction for definably compact semisimple definable groups. Note before proceeding that, by the proof of  \cite[Prop.\,5.1.(iv)]{CP13}, the fundamental group of a definably compact semisimple definable group is finite. Thus, for a definably compact definable semisimple definable group $G$, the kernel of the universal covering map $p_G:\widetilde G\to G$ is finite and therefore the universal covering $\widetilde G$ is definable. Therefore, our last result is a direct consequence of Lemma \ref{L:UpDown}.
	\begin{cor}
	Let $G$ be a definably compact semisimple definable group, let $\widetilde{G}$ be its universal covering and let $Z$ be a finite abelian group. If $\Ext_d(\widetilde{G},Z)\hookrightarrow \Ext(\widetilde G,Z)$ is an isomorphism, then so is $\Ext_d(G,Z)\hookrightarrow \Ext(G,Z)$.\qed
\end{cor}

\end{document}